\documentclass{amsart}[9pt]
\usepackage{amssymb}
\usepackage{tikz-cd}

\newtheorem{thm}{Theorem}[section]
\newtheorem{lem}[thm]{Lemma}

\newtheorem{cor}[thm]{Corollary}
\newtheorem{prop}[thm]{Proposition}

\theoremstyle{definition}  
\newtheorem{example}[thm]{Example}
\newtheorem{definition}[thm]{Definition}

\numberwithin{equation}{section}

\usepackage{pdflscape}

\usepackage{pgf}
\usepackage{tikz}
\usetikzlibrary{arrows,automata}
\usepackage[latin1]{inputenc}
\usepackage{verbatim}

\newcommand{\Z}{\mathbb Z}
\newcommand{\N}{\mathbb N}

\usepackage{color}

\newcommand{\Aa}{{\mathcal A}}
\newcommand{\Bb}{{\mathcal B}}

\newcommand{\Ff}{{\mathcal F}}

\newcommand{\Ll}{{\mathcal L}}

\newcommand{\Pp}{{\mathcal P}}
\newcommand{\Rr}{{\mathcal R}}
\newcommand{\Ss}{{\mathcal S}}

\newcommand{\id}{{\mathbf 1}}

\newcommand{\one}{{\mathbf 1}}

\newcommand{\fl}{\mathfrak f}

\newcommand{\im}{\mbox{\rm im}}

\newcommand{\rset}{I_\theta}

\renewcommand{\a}{{\color{blue}\boxdot}}
\renewcommand{\c}{{\color{red}\boxminus}}
\renewcommand{\b}{{\color{orange}\boxplus}}

\newcommand{\Fac}{\mathrm{Fac}}
\newcommand{\Conj}{\mathrm{Conj}}

\begin{document}
\title{Almost automorphic and bijective factors of substitution shifts}

\author{Alvaro Bustos-Gajardo}
\address{ School of Mathematics and Statistics, The Open University, U.K.  
 \& Facultad de Matem\'aticas, Pontificia Universidad Cat\'olica de Chile, Santiago, Chile
}
\email{abustos@dim.uchile.cl}

\author{Johannes Kellendonk}
\address{
	 Institut Camille Jordan, Universit\'{e} Lyon-1, France
	}
\email{kellendonk@math.univ-lyon1.fr}

\author{Reem Yassawi}
\address{ School of Mathematical Sciences , Queen Mary University of London, U.K.  
}
\email{r.yassawi@qmul.ac.uk}

\thanks{This work was supported by  the
   EPSRC  grant numbers EP/V007459/2 and EP/S010335/1.  The first author received funding from ANID/FONDECYT Postdoctorado 3230159 during the final phase of this work.}


\begin{abstract}
In this article we completely characterise constant length substitution shifts which have an almost automorphic factor, or which have a bijective substitution factor. 
Our approach is algebraic: we characterise these dynamical properties in terms of a finite semigroup defined by the substitution. We characterise the existence of  almost automorphic factors in terms of Green's $\mathcal R$-relation and the existence of bijective factors in terms of Green's ${\mathcal L}$-relation. Our results are constructive. 
\end{abstract}

\maketitle

\section{Introduction}
In this article we are interested in the existence of certain factors for substitutional dynamical systems. Factors with a specific spectral behaviour tell us something about the spectrum of the original system.  In particular, we are interested in characterising when a shift possesses an almost automorphic factor, or when it possesses factors with a  singular component  in their maximal spectral type. We are also interested in factors whose Ellis semigroup we can describe. We investigate these questions for constant length substitution shifts.

More specifically,
we would like to better understand when a substitution shift $(X_\theta,\sigma)$ defined by a primitive aperiodic  length-$\ell$  substitution $\theta$ has a 
factor with the following additional properties. We say that a map is {\em almost injective } if it is somewhere injective, and a dynamical system is {\em almost automorphic} if it is an almost injective extension of a group rotation.
We are interested in characterising two different scenarios:
\begin{enumerate}
\item $(X_\theta,\sigma)$ has a non-equicontinuous factor, which is almost automorphic over the maximal equicontinuous factor of $(X_\theta,\sigma)$, and
\item $(X_\theta,\sigma)$ is an almost injective extension of  a {\em bijective} substitution shift.
\end{enumerate}

In the first scenario, looking for almost automorphic shift factors already solves the problem. This is because the substitution systems we consider have odometers as maximal equicontinuous factors, and an almost automorphic extension of an odometer is necessarily conjugate to a shift \cite[Theorem 6.4]{Downarowicz}. We also note that for the systems we study,  we find an almost automorphic factor $\pi:(X_\theta,\sigma)\rightarrow (Y,\sigma)$ if and only if  the maximal equicontinuous factor map for $(X_\theta,\sigma)$ factors through $\pi$; see Corollary \ref{cor:factor-through}.

 The interest of the second question is twofold. Firstly,
  although ``most" length-$\ell$ substitution shifts should have a singular component in their maximal spectral type, it is only for a subfamily of bijective substitutions that this has been established in some generality; see \cite{Bartlett, Queffelec}. Secondly, at the moment, bijective substitution shifts are one of the few families for which an explicit description of the Ellis semigroup exists \cite{Kellendonk-Yassawi-ESBS}.
  
We answer both questions completely using a combination of two sets of tools. The first tool is algebraic,  the  semigroup $S_\theta$ of a substitution $\theta$; see Definition \ref{def:semigroup}. This semigroup has been extensively used in the case when it is a group $G_\theta$, i.e., when the substitution is bijective. For  example,  $G_\theta$ is used  to characterise the automorphism group \cite{L-M, M-Y}, and, if it is commutative, then $(X_\theta,\sigma,\mu)$ has a singular component in its maximal spectral type \cite{Bartlett, Queffelec}. Also, it is a fundamental building block of the Ellis semigroup of a bijective substitution  \cite{Kellendonk-Yassawi-ESBS}. It is interesting that we use  Green's $\mathcal R$ relation to prove Theorem \ref{thm:characterisation-a.a.-intro}, and  Green's $\mathcal L$ relation to prove Theorems \ref{thm:bijective-inner-intro} and \ref{thm:bijective-intro}.

 The second set of tools is classical and  involves building topologically conjugate versions of $(X_\theta, \sigma)$ using {\em collaring} and {\em k-shifting} (Definitions \ref{def:collaring} and  \ref{def:shifting}). 
  In particular, collaring allows us to control the radius of a putative factor map $F:X_\theta\rightarrow Y$, and $k$-shifting allows us  to compose $F$ with a  ``translation".
  We use both these constructions to limit and manipulate possible factor maps, and this results in theorem statements that are constructive, i.e., given a length-$\ell$ substitution shift, one can explicitly determine whether or not it has an almost automorphic or bijective shift factor.

Let $\theta^{(-l,r)}$ denote the $(-l,r)$-collaring of $\theta$, which is the model of  $\theta$ that we work with if $F:X_\theta\rightarrow Y$ has left and right radius $l$ and $r$. To state our first result,  we distinguish between different families of factor maps. The simplest factor maps are {\em inner encodings}; see Definition \ref{def:inner-encoding}. Equivalently,  a factor map $F:X_\theta\rightarrow Y$ is an  inner encoding if it has radius zero, if $Y$ is substitutional,  and if fixed points are mapped to fixed points. Inner encodings arise whenever there is an  equivalence relation on the alphabet $\mathcal A_\theta$ of the substitution $\theta$ such that if $a\sim b$, then as words, $\theta(a) \sim \theta(b)$. Thus inner encodings define partitions $\mathcal P$ of $\mathcal A_\theta$.  Conversely, given a substitution $\theta$, one can  define a partition $\mathcal P_\theta$, which we call {\em the coincidence partition of $\theta$} (Definition \ref{def:partition}), and 
 which yields an inner encoding of $\theta$, called the {\em inner encoding associated to $\theta$}. This inner encoding generates an almost automorphic shift. However  one cannot guarantee that the inner encoding is aperiodic, so that its shift space is infinite. Our first result is

\begin{thm}\label{thm:characterisation-a.a.-intro}
Let $\theta$ be a  length-$\ell$, primitive aperiodic substitution, with pure base $\tilde \theta$. Then $\theta$ has an aperiodic almost automorphic shift factor if and only if the inner encoding associated to  $\tilde{\theta}^{(-1,1)}$ 
is aperiodic.
\end{thm}
The beauty of this result is that it is quite simple to verify its conditions for a fixed substitution. Both Martin \cite{martin} and later Herning \cite{Herning} have worked on this question, but only in the case where $\theta$ is bijective; see below for a discussion of their results. Note that as a corollary, we can show that there exist substitution shifts for which the maximal {\em tame} factor, \cite{Glasner-Megrelishvili-2023}, equals the maximal equicontinuous factor. For, a tame factor which is not equicontinuous  must be almost automorphic \cite{Huang}, and with  
 Theorem \ref{thm:characterisation-a.a.-intro}, we can give many examples of substitution shifts with no almost automorphic factor.

Our second result characterises, in terms of the semigroup $S_\theta$, when a substitution has a bijective substitution factor. 
We make use of the fact that  $S_\theta$ admits a kernel, i.e., a minimal bilateral ideal, which is a union of minimal left ideals. 
The {\em na\"{i}ve column number} of 
$\theta$ is the rank of any element in the kernel of $S_\theta$.  We first show the following.

\begin{thm}\label{thm:bijective-inner-intro}
Let $\theta$ be  constant length substitution with 
na\"{i}ve column number $c>1$.
Then $(X_\theta, \sigma)$  factors almost injectively onto a bijective inner encoding  on a $c$-letter alphabet
if and only if
$S_\theta$ has a unique minimal left ideal.
\end{thm}

Note that  an almost injective factor onto a bijective substitution exists if and only if  an almost injective factor onto a bijective substitution on $c$ letters exists. Note also that if $c=1$, there is no such factor.

To drop the condition that the bijective factor  comes from an inner encoding,
let  $\theta^{(+k)}$ be the  $k$-shifted extension of $\theta$. This version of $\theta$ is especially useful when one considers factor maps $F$ which translate the  fixed points of $\theta$, from one fibre of the maximal equicontinuous factor to another, i.e., when in  Theorem \ref{thm:ethan}, $\kappa(F)\neq 0$.
We show

\begin{thm}\label{thm:bijective-intro}
Let $\theta$ be an aperiodic primitive constant length-$\ell$ substitution  with na\"{i}ve column number $c>1$ and trivial height. The following are equivalent:
\begin{enumerate} 
\item $(X_\theta,\sigma)$ factors almost injectively onto a bijective  substitution shift on an alphabet with $c$ letters. 
\item There exist $0\leq n,k\leq C$ such that the semigroup $S_{(\theta^n)^{(+k)}}$  contains a unique minimal left ideal. \end{enumerate}
\end{thm}
Moreover, $C$  can be explicitly obtained. In general it is doubly exponential in $\ell$; see the statement of Theorem \ref{thm-bijective-3}.

We discuss prior work concerning Theorem \ref{thm:characterisation-a.a.-intro}. The original result is by Martin \cite{martin}.  
He couched his work in terms of Veech's result on almost isometric systems  \cite[Theorem 7.2]{Veech}, which roughly speaking, 
says that any system with a residual set of distal points can be realised as an inverse limit of alternating isometric and almost automorphic extensions.
 Constant length substitution shifts satisfy the requirements of Veech's theorem, and given their low complexity,  a first natural question arises, which is whether such a shift $(X_\theta,\sigma)$ is already an isometric extension of an almost automorphic shift. Martin investigated this question with the assumption that  $\theta$ is  bijective, i.e., $S_\theta=G_\theta$ is a group. Inside one of his proofs (\cite[Lemma 8.05]{martin}), he made an additional assumption, which is that  any factor map can be assumed to fix fibres with respect to the maximal equicontinuous factor. This renders his work incomplete. Martin goes on to give a necessary and sufficient condition  for the existence of an almost automorphic  factor for a bijective shift, in  \cite[Lemma 8.08]{martin}, and his condition (A) in fact translates to our condition that the minimal sets for the collared $\theta^{(0,1)}$ form a partition. This is almost our statement in 
Theorem \ref{thm:characterisation-a.a.-intro}, although for bijective substitutions the situation may be simpler.

Later, in his thesis, Herning  \cite{Herning}  re-approaches this question; it seems he was unaware of Martin's work. He cites a question of Michael Baake, who asks whether 
any substitution shift  has a subshift factor that is metrically isomorphic to its Kronecker factor, which for substitutions coincides with the maximal equicontinuous factor.  Herning answers this question in the negative by finding  bijective substitutions  that do not have an almost automorphic shift factor. In  \cite[Theorem 4.24]{Herning}, Herning characterises   bijective length-$p$ substitutions of prime length that have an almost automorphic factor, and his work, although more restrictive, does not make the omission that Martin does. Once translated, his characterisation is very similar to  ours.
Our Theorem \ref{thm:characterisation-a.a.-intro} extends these results to characterise when  any primitive length-$\ell$ substitution, not just bijective,  has an almost automorphic factor. 
To complete the connection to Veech's theorem, we mention the elegant result of Lemanczyk and M\"{u}llner \cite{Lemanczyk-Muellner}, notably that any constant length substitution shift has an almost injective extension which has an almost automorphic factor. 

To our knowledge, the question of the existence of bijective factors has not been addressed in the literature.

We summarise the contents of this paper. In Section \ref{Section-basic}, we set the background, fix notation, and define collared, shifted and inner encoded substitutions. We then limit the radii of factor maps for a length-$\ell$ substitution, extending techniques that exist in the literature for invertible factor maps. In Section \ref{sec:-almost-automorphic-result}, we set the stage to prove Theorem \ref{thm:characterisation-a.a.-intro}. To do this, we define {\em outer encodings} of a substitution. These are substitutions whose semigroup is an epimorphic image of $S_\theta$. As such, they do not necessarily give rise to dynamical factors. However,  we show in Proposition \ref{prop:factor-through} that any automorphic factor must factor through the {\em canonical outer encoding}; see Definition \ref{eq:def-rho}. This is the key tool to prove Theorem \ref{thm:characterisation-a.a.-intro}, once we bound possible factor maps using the tools from Section \ref{Section-basic}. In Section \ref{sec:bijective-1}, we prove Theorem
\ref{thm:bijective-intro} in three successive steps. First we prove Theorem \ref{thm:bijective-inner-intro}, which characterises when a substitution has a bijective factor arising from an inner encoding, i.e., a radius zero factor map which sends fixed points to fixed points. Next, in Theorem \ref{thm-main-bijective-2}, we give a characterisation while relaxing the condition that the factor map has radius 0. Finally, in Theorem \ref{thm-bijective-3}, we eliminate the condition that the factor map send fixed points to fixed points. We strive to isolate the requirements on the substitution, noting that generally, purely algebraic results do not need the restriction to primitive aperiodic substitutions. We illustrate with examples throughout.

\section{Preliminaries}\label{Section-basic}
\subsection{Constant length substitutions}\label{sec:basic-def}
A  {\em  length-$\ell$ substitution $\theta$}  is an ordered collection of $\ell$ maps, the so-called column maps $\theta_m:\Aa\to\Aa$, $m=0,\cdots,\ell-1$, 
on a finite set $\mathcal A$, its alphabet.\footnote{Since we do not consider substitutions which are not of constant length, all substitutions will be understood to have constant length.} The substitution
$\theta$ can be understood as a map which associates to a letter $a\in \Aa$ the word $\theta(a):=\theta_0(a)\cdots\theta_{\ell-1}(a)$ and to a word $a_1\cdots a_k$ the    word
\begin{equation}\label{eq-concat}\theta(a_1\cdots a_k) = \theta(a_1)\cdots \theta(a_k),
\end{equation}
of length $k\ell$,
and to the bi-infinite sequences $\cdots u_{-2} u_{-1} u_0 u_1 \cdots$  the bi-infinite sequence  
\[\theta   (\cdots u_{-2} u_{-1} u_0 u_1 \cdots ) :=   \cdots \theta (u_{-2}) \theta (  u_{-1}    ) \theta  (u_{-1})\cdot   \theta (     u_{0})  \theta (    u_{1}) \cdots \, .\]
Here the $\cdot$ indicates the position between the negative indices and the nonnegative indices. 
Powers of $\theta$ (iterated compositions of $\theta$ with itself) are again substitutions and we write ${\theta^k}_m$ for the $m+1$-st map of $\theta^k$. 

A bi-infinite sequence $u$ is {\em $\theta$-periodic } if $\theta^k(u)=u $ for some $k\geq 1$. If $k=1$ then we say that $u$ is a {\em fixed point}.
By taking a power of $\theta$ if necessary, we will assume that each $\theta$-periodic point is $\theta$-fixed.
We say that a finite word is {\em allowed} for $\theta$ if it appears somewhere in $\theta^k(a)$ for some $a\in \Aa$ and some  $k\in\N$.
The {\em substitution shift} $( X_\theta,  \sigma)$ is the dynamical system where the space $X_\theta$  consists of all bi-infinite sequences 
all of whose subwords are allowed for $\theta$.  
We equip $X_\theta$ with the subspace topology of the product topology on $\Aa^\Z$, making the left shift map $\sigma$  a continuous  $\Z$-action.

 In this article 
our techniques are a combination of algebraic arguments involving finite semigroups, and dynamical techniques applied to the dynamical system $(X_\theta, \sigma)$ generated by $\theta$. For the algebraic arguments, very few constraints are imposed on $\theta$. For the dynamical arguments, we
 collect the various properties of substitutions which will play a role.
\begin{itemize}
\item {\em Primitivity.}
We say that $\theta$ is {\em primitive} if there is some 
$k\in \mathbb N$ such that for any $a,a'\in \mathcal A$,
the word $\theta^k(a)$ contains at least one  occurrence of $a'$. For dynamical arguments, substitutions will mostly be assumed primitive. 
Primitivity  of $\theta$ implies that  $X_\theta$ is the shift-orbit closure of any $\theta$-periodic point, and $(X_\theta,\sigma)$ is minimal. If $\theta$ is primitive, then $X_\theta=X_{\theta^n}$ for each $n\in \N$.
 Thus, by considering a power of $\theta$ if necessary, we will assume that all $\theta$-periodic points are $\theta$-fixed.
  \item {\em Aperiodicity.}    
We say that $\theta$ is {\em aperiodic} if $X_\theta$ does not contain any $\sigma$-periodic sequences. This is the case if and only if $X_\theta$ is an infinite space. 
\item {\em Bijectivity.}   
 We say that  $\theta$ is {\em bijective} if all column maps $\theta_m:\Aa\to \Aa$ are bijective.
\end{itemize}

\subsection{The maximal equicontinuous factor of a length-$\ell$ substitution}

Let $\Z_\ell$ denote the  $\ell$-adic integers, i.e., the inverse limit of cyclic groups $ \varprojlim \Z/\ell^n\Z $.  Let  $\Z_{\bar{\ell},h} :=  \varprojlim \Z/\ell^nh\Z $  and let $1 := (\cdots, 0,0,1)$; addition in $\Z_{\bar{\ell},h}$ is performed with carry.   
If $\theta$ is primitive and aperiodic, then Dekking's theorem \cite{dekking} tells us that   $(\Z_\ell, +1)$ is an equicontinuous factor of $(X_\theta, \sigma)$. Furthermore, 
there is an $h$, with $0<h<\ell$, with $h$ coprime to $\ell$, such that 
 $(\mathbb{Z}_{\bar{\ell},h}, +1)$, is  the maximal equicontinuous factor of  $(X_\theta, \sigma)$. The integer $h$ is called the {\em height} of $\theta$, and   we say that  $\theta$ has {\em trivial height} if $h=1$.

We fix the factor map $\pi:X_\theta\rightarrow \Z_\ell$ from a primitive aperiodic length-$\ell$ substitution shift $(X_\theta,\sigma)$ to $(\Z_\ell,+1)$ with which we work in this article. We will specify it by requiring $\pi(u)=0$ if and only if $u$ is a $\theta$-fixed point. We refer the reader to \cite{dekking} for details.

Given the substitution $\theta$, the substitution $\theta^n$ is a length-$\ell^n$ substitution. If $0\leq j \leq \ell^n -1$, we use ${\theta^n}_j$ to denote its $j$-th column map.
The {\em na\"{i}ve column number} of a substitution $\theta$  is defined as the minimal number of distinct letters in the image of a column map of $\theta^{n}$, for some $n$. In other words,
\begin{equation} \label{eq:dolumnnumber}c=c (\theta) := \inf_{j,n} \left\{\lvert {\theta^n}_j(\mathcal A)\rvert  : 0\le j <  \ell^n\right\}.\end{equation}

We say that $\theta$ has a {\em coincidence} if $c=h$. In this case, $(X_\theta,\sigma)$ is {\em almost automorphic}, i.e., an almost injective extension of its maximal equicontinuous factor.  For details, see \cite{dekking}. We remark that our notion of column number is different to Dekking's original definition; however it is in this paper more useful, as it is also in \cite{Lemanczyk-Muellner}, where Lemanczyk and Muellner  show that $h$ divides $c(\theta)$, and that $(X_\theta, \sigma)$ is a somewhere $c (\theta)$-to-one extension of $(\Z_\ell, +1)$. Dekking's definition of the column number of $\theta $ coincides with our definition of na\"{i}ve column number if $\theta$ has trivial height; in this case we will drop the adjective  na\"{i}ve.

\subsection{Collared and shifted substitutions}

  It will be necessary to consider collared substitutions and $k$-shifted substitutions of the substitution $\theta$, which yield shifts that are  topologically conjugate to $(X_\theta, \sigma)$.  The notation we use in the following definition will be useful when we consider recasting factor maps of left radius $l\geq 0$ and right radius $r\geq 0$ as codings.
\begin{definition}\label{def:collaring} 
Let $n=(-l,r)$ where $l,r \geq 0$.
The {\em  $n$-collared extension of $\theta$} is the substitution $\theta^{(n)}$ of the same length whose alphabet consists of the allowed $r+1+l$-letter words of $\theta$ and which is given as follows. Given an allowed word $a_{-l}\ldots a_{r}$ compute $ a'_{-\ell l} \ldots a'_{\ell(r+1)-1 }:=\theta (a_{-l}\ldots a_{r})$ and set
\[ {\theta^{(n)}}_m (a_{-l}\ldots a_{r}) := a'_{m-l}\ldots a'_{m+r}  .\]
 \end{definition}
 If we take $l=r=0$ then we obtain $\theta^{(n)}=\theta$. If $l=0$ then $\theta^{(n)}$ is the so-called $r$-sliding block representation of $\theta$; see
\cite[Section 5.4]{Queffelec}.

\begin{definition}\label{def:shifting} 
Let $0\leq k\leq \ell-1$. The {\em  $k$-shifted extension of $\theta$} is the substitution $\theta^{(+k)}$ of the same length whose alphabet consists of the allowed $2$-letter words of $\theta$ and which is given as follows. Given an allowed word  $a_0 a_{1}$, write $\theta(a_0a_1)=a'_0 \ldots a'_{2\ell-1 }$ and set
\[ {\theta^{(+k)}}_m (a_0 a_{1} ) := a'_{m+k} a'_{m+k+1}    .\]
 \end{definition}
We have $\theta^{(+0)}=\theta^{(0,1)}$.

\subsection{Factors, codes and encoded substitutions}
We say that a shift $(Y,\sigma)$ is a {\em dynamical factor} of the shift $(X, \sigma)$ if there is a continuous, surjective
map $F:X \rightarrow Y$ (the factor map) which intertwines the shifts, $F\circ\sigma = \sigma\circ F$.

The Curtis-Hedlund-Lyndon theorem states that a factor map $F: X\rightarrow Y$ between two shifts, $X\subset \mathcal A^\Z$, $Y\subset \mathcal B^\Z$ is defined by a local rule, that is, given integers $l\geq 0$ and $r\geq 0$ there is a
a surjective map 
$\varphi:\mathcal A^{r+1+l}\rightarrow \mathcal B$, so that 
$$(F(x))_{n}= 
\varphi (x_{n-l}, \dots , x_{n+r})       $$ 
for each $n\in \Z$. The quantities 
 $l, r$        are called the left and right radius of $F$ respectively. If $F$ is radius zero, 
 i.e., $l=r=0$, then  the local rule 
$\varphi:\mathcal A\rightarrow \mathcal B$
of $F$ is called a {\em code}. We use uppercase letters to denote factor maps, and lowercase greek letters to denote local rules. If we are given a local rule $\varphi$ which defines a factor map, we will denote it by $F_\varphi$.

\begin{definition}\label{def:inner-encoding}
Let $\eta$ and $\theta$ be length-$\ell$ substitutions.
We say that $\eta$ is an {\em inner encoding} of $\theta$ if there exists a surjective map
$\beta:\mathcal A_\theta\rightarrow \mathcal A_\eta$ which intertwines the column maps of the substitutions, i.e.,
\begin{center}\begin{tikzcd}

              \mathcal{A}_\theta \arrow[r, "\theta_m"]\arrow[d,"\beta"'] & \mathcal{A}_\theta\arrow[d,"\beta"] \\

              \mathcal{A}_\eta \arrow[r,"\eta_m"] & \mathcal{A}_\eta

\end{tikzcd}\end{center}
commutes.

\end{definition}
If we need more precision then we denote the inner encoding also by the pair $(\eta,\beta)$.
  If $(\eta,\beta)$ is an inner encoding of $\theta$, then $\beta$ is the code of a factor map $F_\beta:X_\theta \rightarrow X_\eta$. However, given an arbitrary code  $\beta:\mathcal A_\theta\rightarrow \mathcal A_\eta$ for a factor map $F_\beta:X_\theta \rightarrow X_\eta$, it will in general not intertwine the substitutions as above. Indeed, the code
  $$\Aa_{\theta^{(+1)}}\ni a_0 a_{1} \stackrel{\tau}\mapsto a_{1}\in\Aa_\theta $$
gives rise to a factor map $F_\tau:X_{\theta^{(+1)}}\to X_\theta$ which is injective and therefore defines a conjugacy, but 
 $\theta$ is not an inner encoding of $\theta^{(+1)}$.

\begin{lem}\label{lem:right-kappa-value}
Consider a  substitution $\theta$ with its $n= (-l,r)$-collaring $\theta^{(n)}$. Let $\imath:\Aa_{\theta^{(n)}}\to \Aa_\theta$ be given by
\begin{align}\label{eq:iota}
 \imath(a_{-l}\ldots a_{r})       := a_0 .\end{align}
Then $(\theta,\imath)$ is an inner coding of $\theta^{(n)}$, and the  factor map $F_\imath$ is injective.
\end{lem}
\begin{proof} Direct computation. \end{proof}

We say that a factor map $F:X_\theta\to X_\eta$ between two substitution shifts of equal length {\em preserves the fixed point fibre} if it maps the fixed points of $\theta$ to fixed points of $\eta$.

\begin{lem}\label{lem:code-factor-kappa}
Let $\eta$ and $\theta$ be  length-$\ell$ substitutions  and let $\beta:\mathcal A_\theta\rightarrow \mathcal A_\eta$ be a code. If
$(\eta,\beta)$ is an inner encoding of $\theta$ then the factor map $F_\beta$ preserves the fixed point fibre.
Hence if $(\eta,\beta)$ is an aperiodic inner encoding of $\theta$, and if $\theta$ has trivial height, then the maximal equicontinuous factor map of $(X_\theta,\sigma)$ factors through $F_\beta$.
 If $\theta$ is primitive then the converse is true as well.
\end{lem}

\begin{proof} {Recall that we can assume, by going over to a power of the substitution if needed, that all $\theta$-periodic points of $\theta$ are fixed.} Let $v.u$ be a fixed point of $\theta$ and hence also of $\theta^m$ for any $m\geq 1$. We denote by
$u=u_0 \ldots$ and $\ldots v_{-1} = v$ the right and left infinite parts.

Suppose that $\eta$ is inner encoded by $\beta$, that is, for each $n$ and $a$, $\beta(\theta^n(a))= \eta^n(\beta(a))$. Then
\begin{align*}F_\beta(v.u) = \lim_n \beta \theta^n(v_{-1}.u_0) &= \lim_n \eta^n(\beta(v_{-1}).\beta(u_0)) \\&=
\eta( \lim_n \eta^n(\beta(v_{-1}).\beta(u_0))) = \eta(F_\beta(v.u)).\end{align*}

The second statement is now immediate.

Conversely, suppose that $F_\beta$ sends fixed points to fixed points, so that $F_\beta(v.u)$ is a fixed point of $\eta^n$ for any $n\geq 1$. If $m < \ell^n$, then $\beta$ maps the $m$-th letter of $u$, which is ${\theta^n}_m(u_0)$, to the $m$th letter of $\beta(u)$, which is ${\eta^n}_m(\beta(u_0))$.  In particular, $\beta$ maps $\theta_{m}(\theta_k(u_0))$ to $\eta_{m}(\eta_k(\beta(u_0))) = \eta_{m}(\beta(\theta_k(u_0)))$. If $\theta$ is primitive all letters arise as ${\theta^N}_k(u_0)$ for some $N$ and $k$ showing that 
$\eta_n(\beta(a)) = \beta(\theta_n(a))$ for all $a\in\Aa_\theta$.
\end{proof}


Any map $\beta:\mathcal A\rightarrow \mathcal B$ between sets defines a partition
$ \Pp_\beta= \{ \beta^{-1}(b) : b\in \mathcal B\}$. We call $ \Pp_\beta$ the {\em partition associated to $\beta$}. 
\begin{lem}\label{lem-code}
Let $\theta$ be a length-$\ell$ substitution on $\mathcal A$.
\begin{enumerate}
\item
If $(\eta,\beta)$ is an inner coding of $\theta$ then the partition $\mathcal P_{\beta}$ associated to $\beta$ satisfies
\begin{align*}\forall m  \, \forall A\in \Pp_\beta\, \exists B\in\Pp_\beta \mbox{ such that } \theta_m(A)\subset B.\end{align*}
\item \label{item-cond}
Conversely, if
 there is a partition $\Pp$ of  $\Aa$ such that
$$\forall m  \, \forall A\in \Pp\, \exists B\in\Pp \mbox{ such that } \theta_m(A)\subset B,$$
then the canonical projection $\beta:\Aa\to \Pp$ defines an inner coding 
$(\eta, \beta)$ of $\theta$ through $\eta_m := \beta \theta_m\beta^{-1}$.
\end{enumerate}
\end{lem}
\begin{proof}
Suppose that $(\eta,\beta)$ is an inner coding of $\theta$. 
Then for all $m$ we have
$\eta_m\circ \beta = \beta\circ \theta_m$.
This implies that for all $b\in\Bb$, $\beta\theta_m$ has the same value on all elements of $\beta^{-1}(b)$.  In other words for any $A\in\Pp_\beta$ we have that $\beta\theta_m(A)$ is a singleton, and this implies that $\theta_m(A)$ must be a subset of a member of  $\Pp_\beta$. 

As for the converse, if $\Pp$ is a partition with the required property then we can define $\Bb:=\Pp$, the code  $\beta:\Aa \to \Bb$ to be the map  that sends $a\in\Aa$ to  the member of $\Pp$ to which it belongs  and for $b\in\mathcal B$, $\eta_m(b)$ is defined to be the member of $\Pp$ which contains $\theta_m(b)$, i.e.,  $\eta_m := \beta \theta_m\beta^{-1}$.
\end{proof}
{If $\mathcal P$ satisfies \eqref{item-cond}  of Lemma \ref{lem-code}, } we call the associated inner encoding  the {\em inner encoding defined by the partition $\Pp$}.

Although not all codes give rise directly to inner encoded substitutions, they induce an inner encoded substitution in the following way:
Let $\theta$ be a substitution on the alphabet $\Aa$ and $\tau:\Aa\to\Bb$ be a code.  
If $\Pp_\tau$ does not satisfy Condition  (\ref{item-cond}) of Lemma~\ref{lem-code}, we can define a finer partition $\tilde\Pp_\tau$ which has this property, notably through the equivalence relation $a\sim b$ if $\forall n\geq 0,0\leq m<\ell^n$, $\tau({\theta^n}_m(a))=\tau({\theta^n}_m(b))$. We denote the associated inner encoded system by $(\eta_\tau,\beta_\tau)$ and call it the {\em inner encoding defined by $\tau$}. The alphabet of $\eta_\tau$ is usually smaller than $\Aa$ and usually larger 
than $\Bb$. There is thus a code $\tau':\Aa_\eta\to\Bb$ which satisfies $\tau=\tau'\circ\beta_\tau$. 
This is summarised in the commuting Figure~\ref{fig:induced-code}.
 \begin{figure}
  \begin{tikzcd}
    \Aa_{\theta} \arrow{r} {\beta_\tau} \arrow[swap]{d}{\tau} &   \Aa_{\eta_\tau} \arrow{dl}{\tau'}     
    \\                
          \Bb &
  \end{tikzcd}
 \caption{The commuting diagram of codes defined by a code $\tau$ and its inner encoding $(\eta_\tau,\beta_\tau)$ .}
 \label{fig:induced-code}
\end{figure}
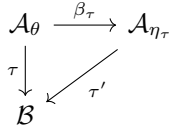

\subsection{Factors of substitution shifts }

Our aim is to show that, up to conjugacy, any dynamical factor of  any primitive, aperiodic, 
constant length substitution ${\theta}$ of trivial height is an inner encoding of a collaring 
$\theta^{(n)}$ of $\theta$, where $n=(-l,r)$ with $l,r\leq 1$. 
Later we will see using   
the suspension construction how the result transposes to the case of any height.

We first state in this context Theorem  \ref{th:factor-bijective},
which is \cite[Theorem 22]{M-Y}, and which we use extensively. It focuses on shift factors which are given by a code. Such shift factors are also called {\em $\ell$-automatic}. 
While
the statement of \cite[Theorem 22]{M-Y} does not mention inner encoded substitutions, a look at its proof shows that the substitution referred to in that theorem is exactly the inner encoded substitution defined by the code. 
Theorem  \ref{th:factor-bijective} uses the additional assumption of pair aperiodicity which we explain. 

We call the letters $a, b$ a {\em periodic pair} if there exists $p=p(a,b)$ and $\ell^p$ such that  $0\leq m<\ell^{p}$ such that  ${\theta^p}_m(a)=a$ and $ {\theta^p}_m(b)=b$.
 We define
 \begin{equation}\label{eq:pair-periodic}
p(\theta):=\mbox{lcm}\{p(a,b): a, b \mbox{ are a periodic pair} \},\end{equation}
and call a substitution $\theta$ with $p(\theta)= 1$ {\em  pair-aperiodic}. $\theta^{p(\theta)}$ is always pair-aperiodic
 \cite{M-Y}. 

\begin{thm}\label{th:factor-bijective}
Let $\theta$ be an aperiodic, pair aperiodic, primitive, length-$\ell$ substitution on $\mathcal A_{\theta}$, of trivial height. 
Let $F_\tau:X_\theta\to Y\subset \Bb^\Z$ be the factor map {defined by a code} $\tau:\mathcal A_\theta \rightarrow \mathcal B$. Consider the inner coding $(\eta_\tau,\beta_\tau)$ of $\theta$ defined by $\tau$ with its accompanying code $\tau':\Aa_{\eta_\tau}\to \Bb$. 
The induced  shift-commuting map $F_{\tau'}:X_{\eta_\tau}\to \Bb^\Z$ is injective and has image $Y$.
In other words $F_{\tau'}$ is a conjugacy between $X_{\eta_\tau}$ and $Y$.
\end{thm}
We have summarise the statement in Figure~\ref{fig-bij-big}. The diagram follows from Figure~\ref{fig:induced-code}. The theorem states that $F_{\tau'}$ is a conjugacy.

 \begin{figure}  \begin{tikzcd}
    X_{\theta} \arrow{r} {F_{\beta_\tau}} \arrow[swap]{d}{F_\tau} &   X_{\eta_\tau } \arrow{dl}{F_{\tau'}}   
       \\                
         Y &
  \end{tikzcd}
 \caption{The commutative diagram for Theorem~\ref{th:factor-bijective}. $F_{\tau'}$ is a conjugacy.}
 \label{fig-bij-big}
\end{figure}

\begin{prop}\label{prop:radius-1}
Let $\theta$ be a  primitive, aperiodic length-$\ell$ substitution, of trivial height, and  
let $F:X_\theta\to Y\subset \Bb^\Z$ be a factor map. There exists an $n=(-l,r)$-collaring $\theta^{(n)}$ and a code $\tau:\Aa_{\theta^{(n)}}\to \Bb$ such that $F = F_\tau\circ F_{\imath}^{-1}$. 
Also, 
if $Y=X_\eta$ is an aperiodic length-$\ell$ substitution shift  
and $F$ preserves the fixed point fibre, 
  then   $l, r \leqslant 1$.
\end{prop}
The statement of this proposition is summarised in Figure~\ref{fig:prop:radius-1}.

  \begin{figure}[h]
  \begin{tikzcd}
    X_{\theta}  \arrow[swap]{dr}{F} &   X_{\theta^{(n)} } \arrow[swap]{l} {F_{\iota}} \arrow{d}{F_\tau}     
        \\               
         & Y
  \end{tikzcd}
  \caption{The commutative diagram for Prop.~\ref{prop:radius-1}. The composition $F_{\beta_\tau}\circ F_\imath^{-1}$ preserves the fixed point fibre.}
  \label{fig:prop:radius-1}
 \end{figure}

\begin{proof}
 Let $l$, $r$ be the left, right  radius of $F$ respectively, and let $n=(-l,r)$. Then there is a code  $\tau:\Aa_{\theta^{(n)}}\to \Bb$ such that $F_\tau:X_{\theta^{(n)}}\rightarrow Y$ is a factor map and such that $F = F_\tau\circ F_{\imath}^{-1}$.  Note that by Lemmas \ref{lem:right-kappa-value}  and \ref{lem:code-factor-kappa},  $F_{\imath}^{-1} $ maps fixed points to fixed points.

To prove the second statement, suppose that $Y=X_\eta$ for some length-$\ell$ substitution  and that $F$ preserves the fibre of fixed points.

  Let $R = \max \{ l, r \}$. We know that $(F (x))_i$ is determined by $x_{[i - R, i
  + R]}$, and thus $(F (x))_{[0, \ell^n)}$ is determined by $x_{[- R, \ell^n +
  R)}$. Thus, if we choose $n \geqslant \lceil \log_{\ell} R \rceil$, this
  ensures $R \leqslant \ell^n$, hence $(F (x))_{[0, \ell^n)}$ is entirely
  determined by $x_{[- \ell^n, 2 \ell^n)}$.

  Fix $n \geqslant \lceil \log_{\ell} R \rceil$. Since $F$ sends fixed points to fixed points, then $F (\theta^n
  (X_{\theta})) \subseteq \eta^n (X_{\eta})$.
    Define $G \colon X_{\theta} \rightarrow X_{\eta}$ as $G := \eta^{- n} \circ
  F \circ \theta^n$. Note that $F \circ \theta^n$ maps $X_{\theta}$ to $\eta^n
  (X_{\eta})$, and the map $\eta^{- n} \colon \eta^n (X_{\eta}) \rightarrow
  X_{\eta}$ is well-defined by recognisability of $\eta$, so $G$ is also
  well-defined. It is not hard to check that $G$ is continuous and $G \circ
  \sigma = \sigma \circ G$, so $G$ is also a factor map. 
  
  Given knowledge
  of $x_{[- 1, 1]}$, we know $(\theta^n (x))_{[- \ell^n, 2 \ell^n)}$ and hence
  $(F \circ \theta^n (x))_{[0, \ell^n)}$ is also determined, as discussed
  above.
  
  Recall that a substitution  $\eta$ is {\em injective} if the map $\eta:\Aa\to \Aa^\ell$ is injective. If $\eta$ is injective, then $(F \circ \theta^n (x))_{[0, \ell^n)}$
  determines $(\eta^{- n} \circ F \circ \theta^n (x))_0$, and thus $G$ has
  left and right radius at most $1$. Also, $G$ must 
   send fixed points to fixed points.
If $\eta$ is non-injective on letters, we can replace it with an
  injectivisation $\tilde{\eta}$ of $\eta$, which is the standard example of an inner encoding of $\eta$, with two letters identified if and only if their images under $\eta$ are equal, see \cite{BDM}. As the natural conjugacy $X_{\eta}
  \rightarrow X_{\tilde{\eta}}$ has radius $0$, this will not change the
  desired result, that $G$ has
  left and right radius at most $1$.

 It remains to show that the radius restriction for $G$ implies the same
  restriction for $F$. Minimality implies that   any
  factor map $X_{\theta} \rightarrow X_{\eta}$ is entirely determined by
  the image of a single point. As the fixed point fibres of substitutions are finite, there can only be finitely many factor maps $F_i:X_\theta\to X_\eta$ which preserve the fixed point fibre. 
Let $\Ff:=\{ F_1, \ldots, F_{k} \}$ be the collection of those.   Let $l_i, r_i$ be the left and right radius of each $F_i$.
  If we define $R = \max \{ l_1, r_1, \ldots, l_{k}, r_{k} \}$ and take $n
  \geqslant \lceil \log_{\ell} R \rceil$, then the above argument shows that $ \eta^{- n} \circ F_i \circ \theta^n$ preserves the fixed point fibre of $\theta$ and hence the
  map {$\Ff \rightarrow \Ff$} given by
  $F_i \mapsto \eta^{- n} \circ F_i \circ \theta^n$ is well-defined. This map
  is a bijection. For, if $ \eta^{- n} \circ F_i \circ \theta^n = \eta^{- n} \circ F_j \circ
       \theta^n $, then $F_i \circ \theta^n (u) = F_j \circ \theta^n
       (u) $ for any point $u \in X_{\theta}$, and in particular for any fixed point $u$ of $\theta$, so that $F_i (u) = F_j (u) $ for $u$ a fixed point. Now minimality implies that the map is injective and since $\Ff$ is finite it is also surjective. 

To conclude, we have seen above that every factor
  map of the form $\eta^{- n} \circ F_i \circ \theta^n$ has left and right radius at most $1$ and so the same must hold for all $F_i$.
\end{proof}

\begin{cor}
\label{cor:factor-bij}
Let $\theta$ be an aperiodic, primitive, length-$\ell$ substitution on $\mathcal A_{\theta}$, of trivial height. 
Let $F:X_\theta\to Y$ be a shift factor. Then there exists an $n=(-l,r)$-collaring $\theta^{(n)}$ with $ l,r \leq 1$, a natural number $p\geq 1$, a    length-$\ell^p$ substitution $\eta$
 and a code $\beta:\Aa_{{\theta^{(n)}}^p}\to \Aa_\eta$ such that $(\eta,\beta)$ is an inner encoding of ${\theta^{(n)}}^p$ and $(X_\eta,
\sigma)$ is conjugate to $(Y,\sigma)$.
\end{cor}

\begin{proof}
By Proposition \ref{prop:radius-1} there is a factor map $F_\tau:X_{\theta^{(n)}}\to Y$  for some collaring $\theta^{(n)}$ of $\theta$ 
such that the
diagram in Figure~\ref{fig:prop:radius-1} commutes.
 
Let $p$ be such that ${\theta^{(n)}}^p$ is pair aperiodic. 
We apply Theorem~\ref{th:factor-bijective} to 
 $F_\tau:X_{{\theta^{(n)}}^p}\to Y$ to obtain the inner encoding $(\eta_\tau,\beta_\tau)$ of ${\theta^{(n)}}^p$. This situation is summarised in 
  the following diagram. Note that $X_{{\theta^{(n)}}^p }$ is equal to $X_{{\theta^{(n)}}}$.
 $$\begin{tikzcd}
    X_{\theta} \arrow{r} {F_{\iota}^{-1}} \arrow[swap]{dr}{F} &   X_{{\theta^{(n)}}^p } \arrow{d}{F_\tau}     \arrow{r} {F_{\beta_\tau}} &X_{\eta_\tau}   \arrow{dl}{F_{\tau'}}   \\                
     & Y
  \end{tikzcd}$$
If $n=(-l,r)$ with $l,r\leq 1$ then we are done, the factor $X_\theta\to Y$ is conjugate to the factor $F_{\beta_\tau}:X_{{\theta^{(n)}}^p}\to X_{\eta_\tau}$  where $(\eta_\tau,\beta_\tau)$ is an inner encoding.

If $n=(-l,r)$ with perhaps $l>1$ or $r >1$ then we need one more step.
 As $F_\imath$ and $F_{\beta_\tau}$ are both obtained from inner encodings, they preserve the fixed point fibres. It follows that  the composition $\tilde F:= F_{\beta_\tau}\circ F_\imath^{-1}:X_\theta \to X_{\eta_\tau}$ preserves the fixed point fibre. We repeat the whole argument above but with $X_{\eta_\tau}$ in place of $Y$. We can apply Proposition \ref{prop:radius-1}, to
  obtain
the commutative diagram
$$
  \begin{tikzcd}
    X_{\theta}  \arrow[swap]{dr}{\tilde F} &   X_{{\theta^{(\tilde n)}}^{\tilde p} } \arrow[swap]{l} {F_{\iota}} \arrow{d}{F_{\tilde \tau}}     
    \arrow{r} {F_{\beta_{\tilde \tau}}} &X_{\eta_{\tilde \tau}}   \arrow{dl}{F_{\tilde\tau'}}   
    \\               
          & X_{\eta_\tau}
  \end{tikzcd}
$$
however this time with $\tilde n = (\tilde l,\tilde r)$ with $\tilde l,\tilde r\leq 1$. This gives us a chain of conjugacies, namely between the factor $F:X_\theta\to Y$ and $\tilde F :X_\theta\to X_{\eta_\tau}$ as we saw above, and then between $\tilde F :X_\theta\to X_{\eta_\tau}$ and $F_{\beta_{\tilde\tau}} :X_{{\theta^{(\tilde n)}}^{\tilde p}}\to X_{\eta_{\tilde \tau}}$. 
 \end{proof}
 We note that the number $p$ in  Corollary  \ref{cor:factor-bij} is bounded by $p(\theta^{(n)})$ in \eqref{eq:pair-periodic}.

  Combining Corollary \ref{cor:factor-bij} and Lemma \ref{lem:code-factor-kappa}, we obtain
  \begin{cor}
\label{cor:factor-through}
Let $\theta$ be an aperiodic, primitive, length-$\ell$ substitution on $\mathcal A_{\theta}$, of trivial height, with maximal equicontinuous factor map $\pi_\theta:X_\theta\rightarrow \Z_\ell$. 
If $(X_\theta, \sigma)$ has an almost automorphic  shift factor, then it has a shift factor $\pi:X_\theta\rightarrow Y$ such that  $\pi_\theta = \pi_Y\circ \pi$, where $\pi_Y:Y\rightarrow \Z_\ell$ is an almost injective factor map.
\end{cor}

\section{The semigroup of a substitution, inner encodings and outer encodings} \label{sec:-almost-automorphic-result}
In this section we completely   characterise the length-$\ell$ substitutions $\theta$ which have a factor that is  almost automorphic over the maximal equicontinuous factor of $(X_\theta,\sigma)$, in Theorem \ref{thm:characterisation} and Corollary \ref{cor:main-thm-2}.
In Section \ref{Section-basic} we have seen that such a factor is conjugate to one which is obtained by an inner encoding of $\theta$ which has a coincidence, provided that $\theta$ has trivial height. We will do this here by using the algebraic structure of the  semigroup $S_\theta$ which we introduce in   Definition \ref{def:semigroup}. We will find that there are always 
inner encodings of $\theta$ which have a coincidence, but the desired almost automorphic factor will exist only if the relevant inner encoding is aperiodic. Finally we explain how to transpose these results to substitutions with non-trivial height.

\subsection{Semigroup preliminaries}
We need some background material on subsemigroups of the semigroup $\Ff(X)$ of maps from $X$ to itself; see also \cite{Clifford-Preston, howie1995fundamentals, Pin}. Here $X$ is just a set and the semigroup product is composition of functions. For our purposes $X$ will be a finite set. 
We denote by $\Pp(X)$ the set of subsets of $X$. 

Let $f:X\to Y$. We denote by $f^{-1}:\Pp(Y)\to \Pp(X)$ the pre-image map but simply write  $f^{-1}(y)$ for $f^{-1}(\{y\})$.
 The map $f$ defines an equivalence relation as $x\sim x'$ if $f(x)=f(x')$. 
We denote the associated partition by $\Pp_f$, that is,
$$\Pp_f =\{f^{-1}(y) : y\in Y\}.$$
The cardinality of $\Pp_f$ equals the rank of $f$, that is the cardinality of its image $\im f$.

Recall  two of  Green's equivalence relations $\Ll$, $\Rr$. They give a first way to approach and organise a semigroup $S$. We say that  $a,b\in S$ are $\Ll$-related, or $\Rr$-related, if they generate the same left, or right ideal, respectively. If $S$ is a group, then these relations coincide with the full relation. Each of the above relations partition the semigroup. 
$\Rr$ is a left congruence and therefore $S/\Rr$ a left-$S$-module. 
$\Ll$ is a right congruence and therefore $S/\Ll$ a right-$S$-module. 

\begin{lem}\label{lem-partition1}
Let $S$ be a semigroup of $\Ff(X)$. If $f,g\in S$ are $\Rr$-related then $\im f = \im g$. If $f,g\in S$ are $\Ll$-related then $\Pp_f = \Pp_g$. 
\end{lem}
\begin{proof}
Let $f$ and $g$ be $\Rr$-related, that is, $f=g$ or there is $f',g'$ such that \ $f = gg'$, $g=ff'$. 
Then clearly $\im f\subset \im g$ and $\im g\subset \im f$. 

Let $f$ and $g$ be $\Ll$-related, that is, $f=g$ or there is $f',g'$ such that \ $f = g'g$, $g=f'f$. 
Then $f^{-1}(x) = g^{-1}({g'}^{-1}(x))$ showing that the members of the partition $\Pp_f$ are unions of members of the partition of $\Pp_g$, i.e.\ the partition $\Pp_g$ is finer than $\Pp_f$. A symmetric argument shows that $\Pp_f=\Pp_g$. 
\end{proof} 

An element $f\in S\subset \Ff(X)$ is {\em completely regular} if it has a normal inverse, that is, there exists $g\in S$ such that $fgf=f$ and $fg=gf$. This implies that $fg$ is an  idempotent. As normal inverses are unique we call $fg$ the idempotent associated to $f$ and denote it $f^0$.

A {\em completely simple} semigroup is a semigroup which has no proper bilateral ideals and which contains an idempotent. 
The {\em kernel} of a semigroup, if it exists, is its smallest bilateral ideal. 
If $X$ is finite then any sub-semigroup of $\Ff(X)$ admits a kernel.
We denote the kernel of $S$ by $\ker S$;
if it contains an idempotent it  is completely simple, and  any element is completely regular.

\begin{lem}\label{lem-partition2}
Let $X$ be a finite set and $S\subset \Ff(X)$. Then $S$ is completely simple if and only if all its functions have the same rank. Moreover, if $f,g\in S$ belong to the same right ideal, then $g^0 f = f$, while if $f,g$ belong to the same left ideal, then $f = f g^0$.
\end{lem}
\begin{proof} Given $n\in\N$, the subset of functions of rank $\leq n$ form a bilateral ideal in $S$. Hence the condition that all functions have the same rank is necessary for simplicity. 

Now suppose that all functions of $S$ have the same rank. Let $f\in S$. Then $\im f = \im f^k$ for all $k\geq 1$. As $X$ is finite there must by a $k>0$ such that $f = f^{k+1}$. Hence $f^{k}=f^0\in S$. The restriction of $f^0$ to $\im f$ is the identity. 
Let also $g\in S$. As $fg\in S$ the argument above shows that there is $k>0$ such that the restriction of $(gf)^k$ to $\im gf = \im g$ is the identity. It follows that $(gf)^k g = g$. Hence $g$ belongs to the bilateral ideal generated by $f$. Choosing $f\in \ker S$ we see that $S=\ker S$ and so is simple. 

Suppose that $f,g\in S$ belong to the same right ideal. Then $f = g g'$ for some $g'\in S$. Hence $g^0f = g^0 g g' = g g' = f$. Similarily, if $f,g\in S$ belong to the same left ideal. Then $f = g' g$ for some $g'\in S$. Hence $f g^0 = g' g g^0 = f$.
\end{proof}
\begin{cor}\label{cor-cp-s}
Let $X$ be a finite set and $S\subset \Ff(X)$. The kernel of $S$ is given by its functions of minimal rank. It is completely simple.
\end{cor}
The Rees structure theorem (\cite[Theorem 3.5]{Clifford-Preston} or \cite[Theorem 2.1]{Kellendonk-Yassawi-ESBS}) tells us that a completely simple semigroup (without zero) is isomorphic to a matrix semigroup
$S \cong I\times G\times \Lambda$ where $I$ indexes the right ideals of $S$, $\Lambda$ indexes the left ideals of $S$, $G$ is a group and 
 there is an $I\times \Lambda$ matrix $M$ such that multiplication in $I\times G\times \Lambda$ is defined as
 \[(i,g,\lambda)(i',g',\lambda')= (i,gM_{\lambda, i'}g', \lambda') \]
 The right ideals are given by $\{i\}\times G\times\Lambda$, $i\in I$, and the left ideals by 
 $I\times G\times\{\lambda\}$, $\lambda\in \Lambda$. In particular one sees that 
 any right ideal intersects any left ideal non-trivally. 
\begin{prop} \label{lem-R-im}
Let $S\subset \Ff(X)$ be a completely simple semigroup and $f,g,\in S$. The following are equivalent.
\begin{enumerate}
\item[R1] $f$ and $g$ are $\Rr$-related.
\item[R2] $f$ and $g$ belong to the same right ideal.
\item[R3] $f$ and $g$ have the same image.
\end{enumerate} 
Moreover, the 
following are equivalent.
\begin{enumerate}
\item[L1] $f$ and $g$ are $\Ll$-related.
\item[L2] $f$ and $g$ belong to the same left ideal.
\item[L3] $f$ and $g$ define the same partition.
\end{enumerate} 
\end{prop}
\begin{proof}
A completely simple semigroup (without zero) is the disjoint union of its right ideals and all the right ideals are simple. As the $\Rr$-class of any element of a simple right ideal is that right ideal, we have equivalence between R1 and R2. It remains to show that R3 implies R2: Suppose that $f$ and $g$ belong to the same left ideal. Then $f g^0 = f$. If also $\im f = \im g$ then $g^0 f = f$, hence  $f$ and $g$ belong to the same right ideal. 
Now suppose that $f$ and $g$ are arbitrary elements satisfying $\im f=\im g$. Then there is $h\in S$ which is in the same right ideal as $g$ and the same left ideal as $f$. The first property implies $\im h=\im g$ and, by the above,  $h$ is also in the same right ideal as $f$. Hence $f$ and $g$ belong to the same right ideal.
\medskip

The equivalence between L1 and L2 is shown as for right ideals. We show that L3 implies L2: Suppose that $f$ and $g$ belong to the same right ideal. Then $g^0 f = f$. If also $\Pp_f = \Pp_g$ then $f g^0 = f$, hence  $f$ and $g$ belong to the same left ideal. The rest of the argument is as for right ideals.
\end{proof}

\begin{lem} \label{lem-epi-ker}
Let $S$ be a semigroup which admits a kernel. 
Let $\Phi:S\to T$ be an epimorphism onto another semigroup $T$. 
The restriction of $\Phi$ to the kernel of $S$ is an epimorphism onto the kernel of $T$. Morever, if $S$ has a unique minimal left ideal then also $T$ has a unique minimal left ideal.
\end{lem}  
\begin{proof}
Any epimorphism preserves ideals: We show this for left ideals. Let $I$ be a left ideal of $S$ and $t\in T$; we ought to show that $t\Phi(I)\subset\Phi(I)$. Since $\Phi$ is onto there is $s\in S$ such that \ $\Phi(s)=t$. Then $t\Phi(I)=\Phi(sI) \subset\Phi(I)$. 

Furthermore, the preimage of an ideal is an ideal: We show this for left ideals. Let $I$ be a left ideal of $T$ and $s\in S$; we ought to show that $s\Phi^{-1}(I)\subset\Phi^{-1}(I)$.
Indeed, $\Phi(s\Phi^{-1}(I)) = t I \subset I$.

Now the above implies that $I:=\Phi(\ker S)$ is the kernel of $T$: Indeed, by the first paragraph it is a bilateral ideal. Let $J\subset I$ be a bilateral ideal of $T$. Then $\Phi^{-1}(J)$ is a bilateral ideal in $S$ by the second paragraph. $\Phi^{-1}(J)$ thus contains $\ker S$. Hence $\Phi(\ker S)\subset J$ showing that $J=I$. 

Suppose now that $\ker S$ is left simple and that $L_1,L_2$ are left ideals of $\ker T$. Then $\Phi^{-1}(L_1)\cap \ker S$ and $\Phi^{-1}(L_2)\cap \ker S$ belong to the same left ideal. As epimorphisms preserve ideals their image under $\Phi$ belongs to the same left ideal. Hence $L_1$ and $L_2$ belong to the same left ideal. Hence $L_1=L_2$.
\end{proof}

\subsection{The semigroup of a length-$\ell$ substitution}
Let $\Aa$ be a finite set.
Let $S\subset \mathcal F(\mathcal A)$.
For $n\in\N$ we define
$S^{(n)}$ to be the family of maps which have rank smaller or equal $n$. If $S^{(n)}$ is not empty then it is a two-sided ideal of $S$.

\begin{definition}\label{def:semigroup}
The semigroup $S_\theta$ of a  length-$\ell$ substitution 
$\theta$ over the alphabet $\Aa$ is the subsemigroup of $\mathcal F(\mathcal A)$ generated by the column maps $\theta_i$, $i=0,\cdots,\ell-1$.
\end{definition}

Recall from  \eqref{eq:dolumnnumber} that the na\"{i}ve column number of $\theta$ is the smallest rank of a product of column maps. From Corollary~\ref{cor-cp-s} we obtain the following. 
\begin{lem}\label{lem:column-number-semigroup} Let $\theta$ be a constant length substitution. The  kernel of $S_\theta$ is $S_\theta^{(c)}$, where $c$ is the na\"{i}ve column number of $\theta$.
\end{lem}

\begin{definition}\label{def:partition}
The {\em minimal sets} of the substitution $\theta$ are the images of the maps of $S_\theta$ of minimal rank. We denote the family of minimal sets by $U_\theta$, i.e.,
$$U_\theta := \{\im f:f\in \ker S_\theta\}.$$
If $U_\theta$ is a cover of $\Aa$, that is, $\Aa= \bigcup_{f\in \ker S_\theta}\im f$, then we call the substitution {\em essentially surjective}.
Given $U_\theta$  we define a relation on the members by $A\sim B$ if $A\cap B\neq \emptyset$. The transitive closure of this relation is an equivalence relation on 
$\bigcup_{f\in \ker S_\theta}\im f$ which defines a partition which we call the {\em  coincidence partition} and denote by $\Pp_\theta$.
\end{definition}

Note that while $S_{\theta^N}\subset S_\theta$ , for any $N\geq 1$, the inclusion may be strict. However their kernels are always equal, by the next lemma.
\begin{lem}\label{lem-cover-power}
Let $\theta$ be a substitution and $N\geq 1$. Then $\ker S_{\theta^N}=\ker S_\theta$. In particular
$U_{\theta}$ and $U_{\theta^N}$ coincide.
\end{lem}
\begin{proof} By definition of the column rank, $\theta$ and $\theta^N$ have the same column rank. Clearly $S_{\theta^N}^{(c)}\subset S_\theta^{(c)}$. To see that the inclusion is surjective, recall  that any element $f$ of $\ker S_\theta$ is completely regular and hence we can factorise $f = f {f^0}^{N-1}$, where we recall that $f^0$ is the idempotent generated by $f$. This shows that $\ker S_{\theta}=\ker S_{\theta^N}$ and immediately implies $U_{\theta}=U_{\theta^N}$.
\end{proof}
\begin{lem} \label{lem-cover} A primitive substitution is essentially surjective.
\end{lem}
\begin{proof}
Suppose that $\theta$ is primitive. Let $a\in \Aa$ and $b\in \im f$ for some $f\in \ker S_\theta$. By primitivity $a$ occurs in $\theta^N(b)$ for some $N$. Hence $a \in \im g f$ for some $g\in S_\theta$. Clearly $gf\in\ker S_\theta$. Hence
$a\in \bigcup_{f\in \ker S_\theta} \im f$.
\end{proof}
The condition of primitivity is sufficient but not necessary: the length $2$ substitution on $\{a,b\}$, $\theta(a) = aa$, $\theta(b) = ab$ is not primitive and $b$ does not belong to a minimal set.
The converse need not be true: the length $1$ substitution $\theta=\id$, where $\id$ is the identity map,  is not primitive, but $U_\theta=\{\Aa\}$. 

\begin{lem}\label{lem:coincidence-1}
Let $\theta$ be an essentially surjective substitution. Then $\Pp_\theta$ is a partition of $\Aa$ which satisfies Condition~(2) of Lemma~\ref{lem-code}. Its associated inner encoding 
$(\eta_{\Pp_\theta},\beta_{\Pp_\theta})$ has na\"{i}ve column number $c=1$. 
\end{lem}
\begin{proof}
Let $A\in U_\theta$, that is, $A=\im g$ for some $g\in \ker S_\theta$. Let $f\in S_\theta$. Then $f(A)\in U_\theta$ as $fg\in \ker S_\theta$.
Now if $A\cap A'\neq \emptyset$ then also $f(A)\cap f(A')\neq \emptyset$.
Thus if two members $A,A'$ of the cover $U_\theta$ belong to the same member of $\mathcal P_\theta$ then also $f(A)$ and $f(A')$ belong to the same member of $\mathcal P_\theta$. This implies Condition~(2) of Lemma~\ref{lem-code} for $\mathcal P_\theta$. Furthermore, if $f$ belongs to the kernel of $S_\theta$ then $f(A)=\im fg = \im f$ and so $f(A)$ is the same for all $A\in U_\theta$. Hence   $\beta_{\Pp_\theta} f \beta_{\Pp_\theta}^{-1}$ has rank $1$, and so  $\eta_{\Pp_\theta}$ has na\"{i}ve column number $1$.
\end{proof}

\begin{definition}\label{def:inner-encoding-associated} Let $\theta$ be an essentially surjective substitution with coincidence partition $\Pp_\theta$. We call the inner encoding defined by $\Pp_\theta$ the {\em inner encoding associated to $\theta$.} We  denote this inner encoding $(\eta_{\Pp_\theta},\beta_{\Pp_\theta})$ by $(\eta_{\theta},\beta_{\theta})$. \end{definition}

As $\eta_\theta$ has column number $1$, then since $h$ divides $c(\theta)$, it has a coincidence and trivial height. Its associated dynamical system is thus almost automorphic. However, it need not  be aperiodic, nor, if it is aperiodic, does it have to have the same maximal equicontinuous factor  as $\theta$. Indeed, if $\theta$ has non-trivial height then the maximal equicontinuous factor of $\eta_\theta$ is strictly smaller than the MEF of $\theta$. In the rest of this section we will show two things: if $\eta_{\theta^{(n)}}$ is periodic for all the collared versions of $\theta$ (we only need $n=(l,r)$ with $l,r\leq 1$) then $X_\theta$ does not admit an aperiodic almost automorphic factor, and if the height is non-trivial we can reduce the task to working with the pure base of $\theta$.
\begin{lem}\label{lem-encoding-power}
Let $\theta$ be an essentially surjective substitution and $N\geq 1$. 
The inner encodings $(\eta_{\theta^N},\beta_{\theta^N})$ and $(\eta_\theta,\beta_\theta)$ associated to $\theta^N$ and $\theta$ satisfy  $\eta_{\theta^N} = {\eta_\theta}^N$ and $\beta_{\theta^N}=\beta_{\theta}$.
\end{lem}
\begin{proof} The statement about the codes,  $\beta_{\theta^N}=\beta_{\theta}$, follows directly from Lemma~\ref{lem-cover-power}. 
Let $0\leq j\leq \ell^N-1$. There are $j_1,\cdots,j_N$ such that \ ${\theta^N}_j = \theta_{j_1}\cdots \theta_{j_N}$. Hence ${\eta_{\theta^N}}_j = \beta_\theta \theta_{j_1}\cdots \theta_{j_N} \beta_\theta^{-1} =  {\eta_\theta}_{j_1}\cdots {\eta_\theta}_{j_N}$.
\end{proof}
\bigskip

We provide an algebraic property of $S_\theta$ characterising column number $1$.
A semigroup is {\em left zero} if every element acts as a zero element when multiplying from the left, i.e.\ $xy=x$ for all $x,y\in S$. If $S$ is completely simple, then it is left zero if and only if the $\Rr$-relation is trivial (equal to the diagonal relation).

\begin{lem}\label{lem-cn-1} A constant length substitution has na\"{i}ve column number $1$ if and only if $\ker S_\eta$ is a left zero semigroup. If this is the case and if $\theta$ is essentially surjective then $\ker S_\eta\ni x\mapsto \im x \in \Aa$ is a bijection. 
\end{lem}
\begin{proof} $\theta$ has na\"{i}ve column number $1$ if and only if $S_\theta$ contains an element of rank $1$ which is equivalent to saying that $\ker S_\theta$ contains exactly the maps of rank $1$ of $S_\theta$. Any collection of rank 1 maps from $\Ff(\Aa)$ forms a left zero semigroup. 

Now, suppose that the na\"{i}ve column number is $c>1$. If $S_\theta$ has more than one minimal left ideal, then $\ker S_\theta$ is not a left zero semigroup. If  $S_\theta$ has a unique minimal left ideal, then, as we will see in Theorem \ref{thm-main-bijective-1}, $\theta$ inner encodes a substitution $\eta$ such that $S_\eta$ is a non-trivial group. In particular $\ker S_\eta = S_\eta$ and since an inner encoding induces an epimorphism from $\ker S_\theta$ to $\ker S_\eta$,   $\ker S_\theta$ cannot be left zero. 

Any rank $1$ map can be identified with the unique element in its image. Hence $\ker S_\eta\ni x\mapsto \im x \in \Aa$ is injective and, if $U_\theta$ is a cover, also surjective.
\end{proof}

\subsection{The canonical outer encoding}

\begin{definition}
Let $\theta$ and $\eta$ be length-$\ell$ substitutions. We say that $\eta$ is {\em outer encoded} by $\theta$ if there is an epimorphism  $\Phi:S_\theta\rightarrow S_\eta$ 
such that $\eta_m= \Phi(\theta_m)$ for $0\leq m \leq \ell -1$.
\end{definition}

We may also say that  $(\eta,\Phi)$ is outer encoded by $\theta$, that  $\eta$ is an outer encoded substitution of $\theta$, or  that $\eta$ is an outer encoding of $\theta$.
Note that if $(\eta,\beta)$ is inner encoded by $\theta$, then $\eta$ is outer encoded  with epimorphism 
$$\Phi_\beta(f) =  \beta f \beta^{-1}.$$
However an outer encoding does not necessarily define an inner encoding or indeed a  dynamical factor; see Example \ref{ex:non-partition}.

\begin{cor}
If $(\eta,\Phi)$ is  outer encoded by $\theta$, then the restriction of $\Phi$  to $\ker S_\theta$ is an epimorphism onto  $\ker S_\eta$.
\end{cor}
\begin{proof} This follows from Lemma \ref{lem-epi-ker} as 
$S_\theta$ is finite and hence admits a kernel. 
\end{proof}

\begin{definition} Let $\theta$ be a  length-$\ell$ substitution with na\"{i}ve column number $c$.  Let $\Bb= \ker(S_\theta)/\Rr$, and ${}^c\Phi:S_\theta\rightarrow \mathcal F(\mathcal B)$ be the morphism defined by 
\begin{align}\label{eq:def-rho}
{}^c\Phi(f)( [x]_\Rr):= [fx]_\Rr.
\end{align}
The {\em canonical outer encoding of $\theta$}  is the substitution ${}^c\theta:  \Bb\rightarrow \Bb^{\ell}    $ defined by
\begin{align}
{}^c\theta_m := {}^c\Phi(\theta_m).  \label{eq:canonical}\end{align}
\end{definition}

The semigroup of a bijective substitution is a group. Therefore $S_\theta= \ker S_\theta$ and  $S_\theta/\Rr $ consists of a single point. Hence the   canonical outer encoding of $\theta$  is  the one letter periodic substitution, which is, of course, a trivial inner encoding.
\begin{example}
The canonical outer encoded substitution  of the collared Thue-Morse system $\theta^{(0,1)}$ is the period doubling substitution, and it is also an inner encoding, see Section~\ref{sec-3.4}. Note that we must take a collared version of $\theta$, as the canonical outer encoding of any bijective substitution  is trivial.
\end{example}
We are interested in finding almost automorphic factors.
The following lemma explains our interest in canonical outer encodings. Note there is no need to assume that $\theta$ is primitive.

\begin{lem}\label{lem:canonical-quasi}
The canonical outer encoding $( {}^c\theta, {}^c\Phi)$  of $\theta$ is primitive and  has column number $1$, i.e., has trivial height  and a coincidence.
\end{lem}
\begin{proof} We will show that there is an element $f\in S_{{}^c\theta}$ whose image contains only one letter; this implies that  ${}^c\theta$ has column number $1$ and therefore a coincidence and trivial height.
By Lemma~\ref{lem-epi-ker}, ${}^c\Phi$ 
restricts to an epimorphism from $\ker S_\theta$ to $\ker S_{{}^c\theta}$. Given $f\in \ker S_\theta$ and $[y]_\Rr\in \ker S_\theta/\Rr$ we have ${}^c\Phi(f)([y]_\Rr) = [fy]_\Rr = [f]_\Rr$. We thus see that the image of ${}^c\Phi(f)$ contains only the letter $[f]_\Rr$.
Hence the column number is $1$.

To see that ${}^c\theta$ is primitive, we need to show that for any two letters $[x],[y]\in\ker S_{{}^c\theta}/\Rr$, there is an $f\in S_{{}^c\theta}$ such that $[fx]=[y]$. 
As any left ideal of $\ker S_\theta$ intersects any right ideal of $\ker S_\theta$, 
any two classes $[x],[y]\in\ker S_{{}^c\theta}/\Rr$ have representatives $x,y$ which belong to the same $\Ll$-class. This means that there is 
$f\in S_{{}^c\theta}$ such that $y=fx$.
\end{proof}

\begin{prop}\label{prop:canonical-not-factor} Let $\theta$ be an essentially surjective 
substitution and $({}^c\theta,{}^c\Phi)$ be the canonical outer encoding of $\theta$. 
Then ${}^c\theta$ is an inner encoding  of $\theta$ if and only if $U_\theta$ 
is a partition of $\Aa$.  In this case, and upon identifying $\ker S_\theta/\Rr$ with $U_\theta$, we have ${}^c\theta = \eta_\theta$, the inner encoding associated with $U_\theta$.
\end{prop}
\begin{proof} Recall that the alphabet of ${}^c\theta$ is $\ker S_\theta/\Rr$. By Lemma~\ref{lem-partition2} the map $\ker S_\theta/\Rr\ni [x]_\Rr\mapsto \im x \in U_\theta$ is a bijection. 

If $U_\theta$ is a partition of $\Aa$ then it coincides with the coincidence partition $\Pp_\theta$ and hence we may identify 
$\beta_\theta:\Aa\to \Pp_\theta$ with the code $a \mapsto [x]_\Rr$ where $x$ is any function from $\ker S_\theta$ which contains $a$ in its image. Say $a=x(b)$. We then have 
${}^c\theta_m(\beta_\theta(a)) = [\theta_m x]_\Rr$ while $\beta_\theta(\theta_m(a)) = \beta_\theta(\theta_mx(b)) =  [\theta_m x]_\Rr$. Hence ${}^c\theta = \eta_\theta$.

Suppose that ${}^c\theta$ is an inner encoding  of $\theta$, i.e.\ there is a code $\beta:\Aa\to  \ker S_\theta/\Rr$ such that $\beta\theta_m ={}^c\theta_m\beta$. 
By Lemma \ref{lem:canonical-quasi},
 ${}^c\theta$ has column number $1$ and so $\ker S_{{}^c\theta}$
contains only rank 1 maps. Hence, for  $g_1, g_2\in \ker S_{{}^c\theta}$ the condition $g_1\neq g_2$ is equivalent to $\im g_1\cap\im g_2=\emptyset$. 
Let $f_i\in ({}^c\Phi)^{-1}(g_i)\cap \ker S_\theta$. If $\im f_1\cap \im f_2\neq\emptyset$ then there are $a_1,a_2\in\Aa$ such that \ $f_1(a_1)= f_2(a_2)$. It follows that $g_1(\beta(a_1))=\beta f(a_1) =\beta f(a_2) =g_2(\beta(a_2))$, hence $\im g_1\cap \im g_2\neq\emptyset$, hence $g_1= g_2$. 
Thus $\im f_1\cap \im f_2\neq\emptyset$ implies ${}^c\Phi(f_1) = {}^c\Phi(f_2)$ which means that $f_1$ and $f_2$ are $\Rr$-related. By Lemma~\ref{lem-partition1} they then have the same image. Thus the elements of $U_\theta$ either coincide or do not intersect. As we assumed that $U_\Aa$ covers $\Aa$ it
is a partition of $\Aa$.
\end{proof}

\begin{example}\label{ex:partition0}  Let
\begin{align*}
  \theta: a  & \mapsto  abcc    \\
     b & \mapsto  badd \\
     c  &\mapsto cacd  \\
     d  &\mapsto dbdc . \\
\end{align*}
This example has $U_\theta=\{\{ a,b\},\{ c,d\}\}$; thus $U_\theta=\mathcal P_\theta$ and ${}^c\theta=\eta_\theta$ is inner encoded. Setting $A=\{a,b\}$ and $C=\{c,d\}$ we obtain
\begin{align*}
  \eta_\theta: A  & \mapsto  AACC    \\
     C  &\mapsto CACC   . \\
\end{align*}
$X_{\eta_\theta}$ is an aperiodic almost automorphic factor of $X_\theta$.
\end{example}
\begin{example}\label{ex:non-partition}
 Consider  the substitution
     \begin{align*}
  \theta:  a  & \mapsto  acaef    \\
     b & \mapsto  bdbde\\
     c  &\mapsto ceccg\\
     d  &\mapsto dfbde \\
     e  & \mapsto  egaef   \\
     f & \mapsto  dfbfg\\
     g  &\mapsto cecge.\\
   \end{align*}
 Its minimal sets are $U_\theta=\{ A=\{a,b,c \}, B=\{c,d,e \},C=\{e,f,g \}       \}$. 
 The canonical outer encoding is $A   \mapsto  ABABC ,   
     B  \mapsto  BCABC,
     C  \mapsto BCACC.$
  The cover  $U_\theta$ does not form a partition. It generates the partition $\mathcal P_\theta = \{ \{a,b,c,d,e,f,g\} \}$ which leads to the periodic substitution 
     \begin{align*}
  \eta_\theta:  D  & \mapsto  DDDDD   \end{align*}   
\end{example}
\begin{lem}
Any inner encoding of an essentially surjective substitution is essentially surjective.
\end{lem}
\begin{proof}
Let $(\eta,\beta)$ be inner encoded by $\theta$, $\beta:\Aa\to\Bb$. Let $b\in \Bb$. As $U_\theta$ covers $\Aa$ there is $a\in \Aa$ and $f\in \ker S_\theta$ such that $\beta(f(a))=b$.
Hence $\beta f \beta^{-1}(\beta(a))=b$. Moreover, $\beta f \beta^{-1}\in \ker S_\eta$. Hence $b\in U_\eta$.
\end{proof}
\begin{prop}\label{prop:factor-through}
Let $\theta$ be a substitution which is essentially surjective.
Let ${}^c\theta$ be the canonical outer encoding of $\theta$.
If  $(\eta,\Phi)$ is outer encoded by $\theta$, and $\eta$ is essentially surjective and has column number $1$, then  $(\eta,\Phi)$ is outer encoded by ${}^c\theta$. \end{prop}
\begin{proof}
Let $({}^c\eta, {}^c\Phi_\eta)$ be the canonical outer encoding of 
$\eta$. We claim that there is a morphism $\varphi:S_{{}^c\theta} \to S_{{}^c\eta}$ such that the diagram
\begin{center}\begin{tikzcd}

                             S_\theta \arrow[r, "{}^{c}\Phi"]\arrow[d,"\Phi"'] & S_{{}^{c}\theta}\arrow[d,"\varphi"] \\

                             S_\eta \arrow[r,"{}^{c}\Phi_\eta"] & S_{{}^{c}\eta}

\end{tikzcd}\end{center}
is commutative. 
By Lemma~\ref{lem-epi-ker} $\Phi(\ker S_\theta)=\ker S_\eta$ and $\Phi$ preserves the $\Rr$-relation. Therefore, given $t,s\in S_\theta$ and $x\in \ker S_\theta$,  $[tx]_{\Rr_\theta}=[sx]_{\Rr_\theta}$ implies $[\Phi(t) \Phi(x)]_{\Rr_\eta}=[\Phi(s) \Phi(x)]_{\Rr_\eta}$ (here $\Rr_\theta$ is the $\Rr$-relation on $S_\theta$ and $\Rr_\eta$ is the $\Rr$-relation on $S_\eta$). Also by Lemma~\ref{lem-epi-ker}, any $y\in \ker S_\eta$ has a pre-image under $\Phi$ in $\ker S_\theta$, so  we see that ${}^c\Phi(t)={}^c\Phi(s)$, which is $[tx]_{\Rr_\theta}=[sx]_{\Rr_\theta}$ for all $x\in \ker S_\theta$, implies ${}^c\Phi_\eta(\Phi(t))={}^c\Phi_\eta(\Phi(s))$, which is $[\Phi(t) y]_{\Rr_\eta}=[\Phi(s) y]_{\Rr_\eta}$ for all $y\in S_\eta$. Thus $\varphi$ is well defined through the formula $\varphi({}^c\Phi(t)) = {}^c\Phi_\eta(\Phi(t))$.

As $\eta$ has column number $1$ and is essentially surjective,  $\ker S_\eta$ can be identified with the alphabet of $\eta$ and therefore
${}^c\Phi_\eta:S_\eta \to \Ff(\ker S_\eta)$ is  injective. Hence $\Phi(t)=(^c\Phi_\eta)^{-1}(\varphi({}^c\Phi(t)))$ showing that it factors through ${}^c\Phi$.
\end{proof}

\begin{definition}
Let $\theta$ be an essentially surjective substitution. Its {\em maximal inner encoding with column number $1$} $\eta$ is an inner encoded substitution  of $\theta$, 
such that any other inner encoding  of  $\theta$ which has column number $1$ factors, {via an inner encoding,}  through $\eta$. \end{definition}

We now show that maximal inner encodings with column number $1$ exist.

\begin{thm}\label{thm:coincidence-existence}
Let $\theta$ be an essentially surjective substitution. Its maximal inner encoding with column number $1$ is the inner encoding $\eta_\theta$ associated to $\theta$. 
In particular, $\theta$ admits an aperiodic inner encoding with column number $1$ if and only if the inner coding associated to $\theta$ is aperiodic.
\end{thm}

\begin{proof}
By Lemma \ref{lem:coincidence-1}
$\eta_\theta$ is an inner encoding with column number $1$. 
Suppose that $\eta$ is an inner encoding of $\theta$ with column number $1$. 
We saw in Proposition \ref{prop:factor-through}  that it must be an outer encoding of the canonical outer encoding  ${}^c\theta$. Furthermore the epimorphism $\Phi:S_{{}^c\theta}\to S_\eta$ restricts to an epimorphism from $\ker S_{{}^c\theta}$ to $\ker S_{\eta}$. Since both have column number $1$, this restriction of $\Phi$ is an epimorphism between left zero semigroups, hence a surjective map from the alphabet of 
${}^c\theta$, which is $U_\theta$, to the alphabet of $\eta$. On the other hand, the code from $\Aa$ to the alphabet of $\eta$ is given by a partition $\Pp$. Any member of $U_\theta$ must be a subset  of an element of $\Pp$.
  Thus any element of $\Pp_\theta$ is a subset of an element of $\Pp$, and 
$\eta$ is a inner encoding of the inner encoding $\eta_\theta$ defined by $\Pp_\theta$.  
\end{proof}

\label{ex:partition} 

\subsection{Trivial height}
Recall that a substitution has a coincidence if and only if its height coincides with its na\"{i}ve column number. In this subsection we consider the case in which $\theta$ has height $1$. This allows us to exploit the above results about inner encoded substitutions with column number $1$ for the analysis of factors which are almost automorphic.

\begin{thm}\label{thm:characterisation}
Let $\theta$ be a primitive substitution of trivial height. Then $X_\theta$ has an aperiodic almost automorphic shift factor over its maxmal equicontinuous factor if and only if one of the inner encodings associated to  $\theta^{(n)}$, $n=(-l,r)$, $0\leq l,r\leq 1$, 
is aperiodic.
\end{thm}

{Note that Theorem  \ref{thm:characterisation} and Lemma \ref{lem:code-factor-kappa} imply that if $X_\theta$ has an aperiodic almost automorphic factor, then its maximal equicontinuous factor map factors through this almost automorphic factor.}

We remark that it is sufficient to check whether the inner encoding associated  $\theta^{(-1,1)}$  is aperiodic, as all other $\theta^{(n)}$ are inner encodings of   $\theta^{(-1,1)}$. However if one suspects that $X_\theta$ has an almost automorphic factor, in practice it is easier to check whether  $X_\theta$ or $X_{\theta^{(0,1)}}$ give almost automorphic factors, especially if doing these computations by hand.

\begin{proof} 
By Cor.~\ref{cor:factor-bij} any factor $X_\theta\to Y$ is conjugate to an inner encoding of ${\theta^{(n)}}^p$ for some $|n|\leq 1$ and $p\geq 1$.  By Theorem 
\ref{thm:coincidence-existence}
the maximal inner encoding
of $(\theta^{(n)})^{p}$ with column number $1$ is the inner encoding defined by $\Pp_{(\theta^{(n)})^{p}}$. By Lemma~\ref{lem-encoding-power} the inner encoding of $(\theta^{(n)})^{p}$ defined by $\Pp_{(\theta^{(n)})^{p}}$ is the $p$-th power of the  inner encoding of $\theta^{(n)}$ defined by $\Pp_{\theta^{(n)}}$. 
We can therefore conclude that $X_\theta$ has an aperiodic almost automorphic shift factor over its maximal equicontinuous factor if and only if one of the inner encodings defined by  $\mathcal P_{\theta^{(n)}}$, $n=(-l,r)$, $l,r\leq 1$, has maximal equicontionuous factor $(\Z_\ell,(+1))$. Indeed, as the inner encoding has column number $1$ its shift is almost automorphic, but if it is periodic, its maximal equicontinuous factor is finite, whereas if it is aperiodic its maximal equicontinuous factor must be $(\Z_\ell,(+1))$, which coincides with the 
maximal equicontinuous factor of $\theta^{(n)}$, as $\theta^{(n)}$ has trivial height.
\end{proof}

\begin{example}
     We return to Example \ref{ex:non-partition} which has height $1$. 
         We saw that $\theta$ has no nontrivial inner encoding with a coincidence. By the comment after Theorem \ref{thm:characterisation}, it is enough to check whether $X_{\theta^{(-1,1)}}$ has an almost automorphic factor. We find

  \begin{align*}
\theta^{(-1,1)}\colon 0 & \mapsto W7I68 &
1 & \mapsto YSf1Q &
2 & \mapsto LFU3C \\
3 & \mapsto LFU3E &
4 & \mapsto LFV4b &
5 & \mapsto e051R \\
6 & \mapsto h7I69 &
7 & \mapsto oSf1P &
8 & \mapsto h7KAj \\
9 & \mapsto h7KAl &
A & \mapsto h7KAn &
B & \mapsto H2B3D \\
C & \mapsto MSf1O &
D & \mapsto MSf1Q &
E & \mapsto MSf1R \\
F & \mapsto LFV4a &
G & \mapsto T051P &
H & \mapsto U2B3D \\
I & \mapsto W7I69 &
J & \mapsto W7I6A &
K & \mapsto W7I69 \\
L & \mapsto XFU3D &
M & \mapsto YSf1P &
N & \mapsto XFV4Z \\
O & \mapsto XFV4a &
P & \mapsto XFV4b &
Q & \mapsto XFV4c \\
R & \mapsto XFV4d &
S & \mapsto W7KAj &
T & \mapsto e051P \\
U & \mapsto g2B3D &
V & \mapsto g2B3D &
W & \mapsto h7I6A \\
X & \mapsto iFU3D &
Y & \mapsto oSf1P &
Z & \mapsto h7KAj \\
a & \mapsto h7KAk &
b & \mapsto h7KAl &
c & \mapsto h7KAm \\
d & \mapsto h7KAn &
e & \mapsto G051P &
f & \mapsto G051R \\
g & \mapsto H2B3D &
h & \mapsto J7I6A &
i & \mapsto LFU3D \\
j & \mapsto MSf1N &
k & \mapsto MSf1O &
l & \mapsto MSf1P \\
m & \mapsto MSf1Q &
n & \mapsto MSf1R &
o & \mapsto MSf1P.
\end{align*}
Furthermore, $\mathcal P_{\theta^{(-1,1)}    }=\{\alpha,\beta,\gamma\}$ where 
 \begin{align*}& \alpha=\{1, 3, 4, 6, 8, 9, A, C, D, E, N, O, P, Q, R, Z, a, b, c, d, j, k, l, m, n\},\\
&\beta= \{5, B, G, H, I, J, K, L, M, T, U, V, W, X, Y, e, f, g, h, i, o\},\mbox{ and }\\
&\gamma=\{0, 2, 7, F, S\}.
\end{align*}
The maximal inner encoding with a coincidence is the substitution 
 \begin{align*}
 \eta_{\theta^{(-1,1)}}: \alpha & \mapsto \beta \gamma\beta\alpha \alpha \\
  \beta & \mapsto \beta \gamma\beta\alpha \alpha \\
   \gamma & \mapsto \beta \gamma\beta\alpha \alpha \end{align*}
which is periodic. We conclude that $X_\theta$ does not have an almost automorphic factor. 
 \end{example}
 \subsection{Non-trivial height} Suppose that our substitution $\theta$ has non-trivial height. While we can still apply  Theorem~\ref{thm:coincidence-existence} to find the maximal encoded substitution with column number $1$, and this one will give an almost automorphic factor of $X_\theta$, the resulting substitution $\eta_\theta$ must have height $1$ and thus its  MEF is strictly smaller than that of $\theta$, as Example \ref{ex:wrong-aa-factor} shows. 
 This is not what we are after. 
 
 \begin{example}\label{ex:wrong-aa-factor} To highlight the difference between working with or without the pure base, consider 
\begin{align*}
\theta:   a  & \mapsto  aba   \\
     b &  \mapsto  bac \\
     c & \mapsto cab
\end{align*}     
which has height $2$. Its semigroup is  
$$S_\theta = \{\id,\tau_{bc}\}\sqcup \{p_b,p_c,\bar p_b,\bar p_c\}$$
where $\tau_{bc}$ exchanges $b$ with $c$ and leaves $a$ fixed, $p_b$ is the idempotent of rank $2$ mapping $c$ to $b$, $p_c$ is the idempotent of rank $2$ mapping $b$ to $c$, $\bar p_b$ is $p_b$ followed by an exchange of $a$ and $b$ and $\bar p_c$ is $p_c$ followed by an exchange of $a$ and $c$. The kernel $\ker S_\theta$ is thus given by  rank $2$ maps $\ker S_\theta = LZ_2\times\Z/2\Z$, where  $LZ_2$ is the left zero semigroup on two letters. We see that $\theta$'s height equals its na\"{i}ve column number and hence the substitution shift  is a somewhere injective extension of its maximal equicontinuous factor. So in principle we don't need to look at an almost automorphic factor, as the substitution shift already is almost automorphic. But it is instructive to apply the construction of the canonical outer encoding, as the result is a primitive almost automorphic substitution with trivial height. Indeed, if we denote ${}^c b = [p_b]$, the $\Rr$-class of $p_b$ and ${}^c c = [p_c]$ that of $p_c$ we find that the canonical outer encoded substitution is given by 
 \begin{align*}
{}^c\theta:   
     {}^c b &  \mapsto  {}^c b {}^c b {}^c c \\
       {}^c c &  \mapsto  {}^c c {}^c b {}^c b \end{align*}     
Furthermore, the cover $U_\theta$ is given by $\{\im p_b,\im p_c\} = \{\{a,b\},\{a,c\}\}$ which is not a partition, hence the canonical outer encoded substitution is not inner encoded. The coincidence partition is trivial $\Pp_\theta = \Aa$.

\end{example}

Thus instead we will work with the pure base of $\theta$.

We recall from  \cite{dekking} the following results:
If $\theta$ is primitive and has height $h$, then  there exists a $\sigma^h$-periodic clopen partition of $X_\theta$, $X_\theta= \bigsqcup_{k\in\Z/h\Z}  X_\theta^k$, and $\sigma(X_\theta^k) = X_\theta^{k+1}$.
Moreover, $X_\theta^0$ is invariant under $\theta$ and there exists a subset $\Aa'\subset\Aa^h$ such that $X_\theta^0$ consists precisely of the sequences $x\in X_\theta$ for which $x_0\dots x_{h-1}\in\Aa'$. Define a substitution $\theta'$ on $\Aa'$ as follows: 
Given $a_0\dots a_{h-1}\in\Aa'$ compute $\theta(a_0\dots a_{h-1})=a'_0\dots a'_{h\ell-1}$ and set
$$\theta'_k(a_0\dots a_{h-1}) = a'_{kh}\dots a'_{(k+1)h-1}.$$
$\theta'$ is called the pure base of $\theta$. It corresponds to the restriction of $\theta$ to $X_\theta^0$ but expressed in the alphabet $\Aa'$.

The suspension of a $\Z$-action $(X,\varphi)$ with $\Z/h\Z$ is the space $X\times \Z/h\Z$ equipped with the $\Z$-action 
\begin{equation}\label{eq:height-rep}
 T_{\varphi}(x,i):=
\begin{cases}
  (x,i+1)      & \text{ if  } 
0\leq  i<h-1    \\
    (\varphi(x), 0)   & \text{ if } i=h-1
\end{cases}
\end{equation} 

The shift $(X_\theta,\sigma)$ is, for a substitution of height $h$, topologically conjugate to a  suspension of the substitution shift $(X_{\theta'}, \sigma)$ with the finite group $\Z/h\Z$. 
Here we have denoted the shift action on $X_{\theta'}$ by $\sigma'$ in order to distinguish it from the shift action $\sigma$ on $X_\theta$. The conjugacy is given by
$$X_{\theta'}\times \Z/h\Z\ni (x',i) \mapsto \sigma^i(x)\in X_{\theta},$$
where on the left hand side $x'$ is  a sequence of letters from $\Aa'$, that is, a sequence of allowed (for $\theta$) words of length $h$ whereas on the right hand side 
$x$ is the sequence of letters from $\Aa$ that one obtains when one interprets $x'$ as a sequence of letters in $\Aa$. Note that 
$(x',i)\mapsto (\sigma'(x'),i)$ on the left corresponds to $\sigma^i(x)\mapsto \sigma^{h+i}(x)$ on the right.
The suspension construction is functorial and immediately implies:
\begin{itemize}
\item If $F:(X_{\theta'},\sigma')\to (Y,\varphi)$ is a factor map then  $F\times\id:(X_{\theta'}\times\Z/h\Z,T_{\sigma'})\to (Y\times\Z/h\Z,T_{\varphi})$ is a factor map and any factor map of 
$(X_{\theta'}\times\Z/h\Z,T_{\sigma'})$, up to a rotation, arises in this way. 
In particular the MEF of $(X_\theta,\sigma)$ is conjugate to $(\Z_\ell\times \Z/h\Z, T_{+1})$.
\item $(Y,\varphi)$ is almost automorphic if and only if $(Y\times\Z/h\Z,T_{\varphi})$ is almost automorphic.
\end{itemize}

 Recall that a topological dynamical system $(X, T )$ is a minimal almost automorphic extension of an odometer if and only if it is topologically conjugate to a shift \cite[Theorem 6.4]{Downarowicz}.  Combining this with the remarks above, we obtain the following.

     \begin{prop}\label{lem:almost-automorphic-height.}
     Let $\theta$ be a  length-$\ell$, primitive aperiodic substitution, with pure base $\theta'$. Then $(X_\theta, \sigma)$ has an aperiodic almost automorphic shift factor above its maximal equicontinuous factor if and only if $(X_{\theta'}, \sigma')$ has an aperiodic almost automorphic shift factor above its maximal equicontinuous factor.
     \end{prop}

   We describe how to construct the desired almost automorphic factor  of $X_{\theta}$  when $X_{\theta'}$ has an aperiodic almost automorphic shift factor $X_{\eta'}$ via the map $F':X_{{\theta'}}\rightarrow  X_{\eta'}$.
       $\eta'$ necessarily has height $1$.  Define a new length-$\ell$ substitution $\eta$ with alphabet $\mathcal{A}_\eta=\{a_j \::\: a\in\mathcal{A}_{\eta'},1\le j \le h\}$ as follows. Define $\mathfrak i: \mathcal A_{\eta'} \rightarrow \mathcal A_\eta^{h} $ by $\mathfrak i(a)= a_1 \dots a_h$. Now let $\eta$ be the unique length-$\ell$ substitution which satisfies $\eta\circ \mathfrak i = \mathfrak i \circ \eta'$.
That is, we ``split'' each $a\in\mathcal{A}_{\eta'}$ into $h$ different letters $a_1,\dotsc,a_h$ in such a way that the concatenation of the length $\ell$ words $\eta(a_1)\cdots\eta(a_h)$ is the word obtained from ${\eta'}(a)$ by splitting every letter. As defined, $\eta$ has height $h$ and pure base ${\eta'}$. Thus,  $X_\eta$ is also a suspension of $X_{\eta'}$ over $\mathbb{Z}/h\mathbb{Z}$.  It can be seen that $X_\eta $ is almost automorphic over $(\Z_{\bar\ell,h}, +1)$. See Example \ref{ex:height-bijective} for such a construction.

     \begin{cor}\label{cor:main-thm-2}
     Let $\theta$ be a  length-$\ell$, primitive aperiodic substitution, with pure base $\tilde\theta$. Then $X_\theta$ has an aperiodic almost automorphic shift factor if and only if the inner encoding defined by  $\mathcal P_{\tilde{\theta}^{(-1,1)}}$
is aperiodic.

     \end{cor}

\subsection{Two-letter substitutions of bijective substitutions and their canonical outer encodings}\label{sec-3.4}
In this section we apply our results above to study when a bijective substitution shift has an almost automorphic factor. In the process we revisit the work in \cite[Section 4]{Kellendonk-Yassawi-ESBS}. We start with bijective substitutions of trivial height.
Remember that for bijective substitutions, $S_\theta$ is a group and so ${}^c\theta$ is trivial, being defined on a one letter alphabet. This does not mean that a bijective substitution shift does not admit an almost automorphic substitutional factor, but only that these may only be seen when working with the collared versions of the substitution. We will work with $\theta^{(0,1)}$ and, 
in case that $X_{\theta^{(0,1)}}$ does not have an almost automorphic factor, verify whether or not $\theta^{(-1,1)}$ does. 

For $\theta$ a bijective substitution on $\Aa$ let $\Aa^{(2)}$ be the set of allowed two-letter words for $\theta$. Recall that the 2-collared substitution  
$\theta^{(0,1)}$ associated to $\theta$ is the substitution on $\Aa^{(2)}$ of the same length given by
$$\theta^{(0,1)}_m(a,b) = (\theta_m(a),\theta_{m+1}(a)), \quad 0\leq m<\ell-1, \quad 
\theta^{(0,1)}_{\ell-1}(a,b) = (\theta_{\ell-1}(a), \theta_{0}(b))$$
We assume that all $\theta$-periodic points are fixed,  so that  $\theta_0=\theta_{\ell-1}=\id$; thus all maps $\theta^{(0,1)}_m$ with $m<\ell-1$ have rank $c=|\Aa|$ while $\theta^{(0,1)}_{\ell-1}$ is equal to the identity on $\Aa^{(2)}$ and hence has rank equal to $|\Aa^{(2)}|$.
We know that the na\"{i}ve column number of a substitution is a conjugacy invariant and so  the na\"{i}ve column number of $\theta^{(0,1)}$ must also be $c$. Hence we see that the kernel of $S_{\theta^{(0,1)}}$ is generated by $\theta_m \mathrm{pr_1}\times\theta_{m+1} \mathrm{pr_1}$, $m<\ell-1$, where $\mathrm{pr_1}:\Aa^{(2)}\to \Aa$ is the projection onto the first factor. Recall that the right ideals of $\ker S_{\theta^{(0,1)}}$ are in one-to-one correspondence to the images of these maps. As $\theta_m$ is bijective, the image of  $\theta_m \mathrm{pr_1}\times \theta_{m+1} \mathrm{pr_1}$ coincides with that of $\mathrm{pr_1}\times \theta_{m+1}\theta_m^{-1} \mathrm{pr_1}$ and so uniquely is determined by the map $\theta_{m+1}\theta_{m}^{-1}$. We thus see that the set of right ideals is in one-to-one correspondence the set 
$$I_\theta:=\{\theta_{m+1}\theta_{m}^{-1}|m=0,\cdots,\ell-2\}$$ 
which plays a prominent role in the description of the Ellis semigroup of the substitution shift $(X_\theta,\sigma)$ 
and is also called the $R$-set of the substitution \cite{Kellendonk-Yassawi-ESBS}. 
We have shown the first part of 
\begin{cor}\label{cor:bijective}
The alphabet of  the canonical outer encoding  {${}^c\theta^{(0,1)}$} of $ \theta^{(0,1)}$ 
 can be identified with $I_\theta$. 
Under this identification it is given by 
$${{}^c\theta^{(0,1)}}_m(\theta_i\theta_{i-1}^{-1}) = \theta_{m+1}\theta_{m}^{-1},\quad 
0\leq m< \ell-1,\quad {{}^c\theta^{(0,1)}}_{\ell-1}(\theta_i\theta_{i-1}^{-1})=\theta_i\theta_{i-1}^{-1}$$
It is an inner encoding of $\theta^{(0,1)}$ if and only if
for all $f,g\in I_\theta$ and $\forall a\in\Aa$:  
$f(a)\neq g(a)$.
Thus  the following is a necessary condition for  ${}^c\theta^{(0,1)}$ to be an inner encoding:
\begin{equation}\label{eq-nec-code}
|I_\theta| \times |\Aa| = |\Aa^{(2)}|
\end{equation}
and so the code of the inner encoding is a $|\Aa|$-to-1 map.
\end{cor}
\begin{proof}
Recall that the canonical outer encoding is an inner encoding if the images of the maps $\theta_m \mathrm{pr_1}\times\theta_{m+1} \mathrm{pr_1}$, $m<\ell-1$, either coincide or do not overlap. As the images of $\theta_m \mathrm{pr_1}\times\theta_{m+1} \mathrm{pr_1}$ and $\theta_{m'} \mathrm{pr_1}\times\theta_{m'+1} \mathrm{pr_1}$ coincide if and only if 
$\theta_m\theta_{m-1}^{-1} = \theta_{m'}\theta_{m'-1}^{-1}$ this is exactly the condition stated. If this is the case, then the number of letters of $\theta^{(0,1)}$ is equal to the maximal choice of distinct maps $\theta_m \mathrm{pr_1}\times\theta_{m+1} \mathrm{pr_1}$, $1\leq m <\ell-1$ (the size of $I_\theta$) times the size of the image of one of them, which is $|\Aa|$. 
\end{proof}

We go through the examples in \cite{Kellendonk-Yassawi-ESBS}.

The canonical outer encoding of  any $2$-letter bijective substitution is an inner encoded substitution, because for those $I_\theta = \{\id,\fl\}$, where $\fl$ interchanges $a$ with $b$. 
The simplest example is the Thue-Morse substitution 
\begin{align*}
\theta:   a  & \mapsto  abba   \\
     b &  \mapsto  baab
\end{align*}
whose semigroup is $S_\theta = \Z/2\Z$. The alphabet of $\theta^{(0,1)}$ is
 $$\Aa^{(2)}=\{o=ab,\bar o=ba,e=aa,\bar e = bb\}$$ 
 and $$S_{\theta^{(0,1)}}=\{\id,\omega\}\sqcup \{\Pi_0,\bar\Pi_0,\Pi_e,\bar\Pi_e\}$$
 where $\omega$ is the order $2$ symmetry exchanging $o$ with $\bar e$ and $e$ with $\bar o$, 
 $\Pi_o$ the rank $2$ idempotent mapping $e$ to $o$ and $\bar e$ to $\bar o$,
 $\Pi_e$ the rank $2$ idempotent mapping $o$ to $e$ and $\bar o$ to $\bar e$, and 
 $\bar \Pi_\bullet$ is $\Pi_\bullet$ followed by barring the letter. The kernel contains the four maps of rank $2$ and is isomorphic to $\ker S_{\theta^{(0,1)}}\cong LZ_2\times\Z/2\Z$, the left zero semigroup of two elements times the group $\Z/2\Z$. We denote by $[o]$ the $\Rr$-class of $\Pi_0$ and by $[e]$ the $\Rr$-class of $\Pi_e$.
 It is now easily seen that 
 ${}^c\phi(\id) = \id$, ${}^c\phi(\omega) = \fl$, ${}^c\phi(\Pi_o)$ the projection onto $[o]$, ${}^c\phi(\Pi_e)$ the projection onto $[e]$, and ${}^c\phi(\bar \Pi_\bullet)=\fl \,{}^c\phi(\Pi_\bullet)$. With this we find that the column maps of ${}^c\theta^{(0,1)}$ are given by 
 \begin{align*}
{}^c\theta^{(0,1)} :  o  & \mapsto  oeoo   \\
     e &  \mapsto  oeoe
\end{align*}
This result is well known, $ {}^c\theta^{(0,1)} $ is the period doubling substitution.

Our next example has three letters,
\begin{align*}
\theta:   a  & \mapsto  abcca   \\
     b &  \mapsto  babab \\
     c &         \mapsto       ccabc
     \end{align*}
It does not satisfy (\ref{eq-nec-code}), as $\Aa^{(2)}$ has five letters and five is a prime number. A further investigation of $\theta^{(-1,1)}$ shows, that also its canonical outer encoding has a trivial  coincidence partition. Thus the substitution shift does not have an aperiodic almost automorphic factor.

In our second example\footnote{This example  has generalised height equal to 2 but trivial classical height.}
\begin{align*}
\theta:   a  & \mapsto  abacaaa \\
     b &  \mapsto   babbbcb \\
     c &         \mapsto    cccacbc
    \end{align*}
we have $\rset = \{\theta_1,\theta_3,\theta_5\}$ with $\theta_1 = \begin{pmatrix} b\\a\\c\end{pmatrix}$,
$\theta_3 =  \begin{pmatrix}c\\b\\a \end{pmatrix}$, and 
$\theta_5 =\begin{pmatrix}a\\c\\b \end{pmatrix}$. 
We see that the partition of $\Aa^{(2)} = \Aa^2$ is given by $A=\{ab,bc,cc\}$, $B=\{ac,bb,ca\}$, $C=\{aa,bc,cb\}$ leading to the inner encoded substitution
\begin{align*}
\eta:   A  & \mapsto AABBCCA \\
     B &  \mapsto   AABBCCB \\
     C &         \mapsto    AABBCCC.
  \end{align*}
which is hence an almost automorphic substitutional shift factor of $X_\theta$.

In our third example $\theta$ is the substitution
\begin{align*}
\theta:   a  & \mapsto  abc   \\
     b &  \mapsto  bca \\
     c &         \mapsto       cab
     \end{align*}
has $\theta^3$-fixed points, so we consider the  third power $\theta^3$.
We have $\rset = \{\one,\omega,\omega^{2}\}$ leading to the partition  
of $\Aa^{(2)} = \{ab,bc,ca,ac,ba,cb\}$ given by $A=\{ab,bc,ca\}$, $B=\{ac,ba,cb\}$.
This yields the substitution $\eta$
\begin{align*}
\eta:   A  & \mapsto AABAABAAA \\
     B &  \mapsto   AABAABAAB, \\
      \end{align*}
so that $X_\eta$ is an almost automorphic factor of $X_\theta.$

As
our last example from \cite{Kellendonk-Yassawi-ESBS} has nontrivial height, we go through the required details carefully.
\begin{example}\label{ex:height-bijective}
 This substitution has four letters and $G_\theta$ equals the dihedral group $D_4$:
\begin{align*}
\theta:   a  & \mapsto  abadcba    \\
     b &  \mapsto  badcbab \\
     c &         \mapsto       cdcbadc\\
     d& \mapsto  dcbadcd
     \end{align*}and it  has  height 2. We have $\rset = \{\theta_1,\rho\}$ where $\rho=  \theta_2\theta_1^{-1}$ (This corrects the mistake in \cite{Kellendonk-Yassawi-ESBS}   which stated that         $\rho=    \theta_1\theta_2$). The alphabet of the canonical outer encoding ${}^c\theta^{(0,1)}$ thus has $2$ letters. But $\Aa^{(2)}$ contains only $6$ letters. Therefore \eqref{eq-nec-code} tells us that  ${}^c\theta^{(0,1)}$ cannot be an inner encoded substitution.

To find out whether the shift generated by this substitution  has an almost automorphic factor, then, as it has height 2, we move to its pure base by Lemma \ref{lem:almost-automorphic-height.}. This is given by:

\begin{align*}\tilde{\theta}\colon 0 & \mapsto 3 0 1 0 1 0 2\\
              1 & \mapsto 2 1 0 1 0 1 3\\
              2 & \mapsto 2 1 0 2 1 0 2\\
              3 & \mapsto 3 0 1 3 0 1 3\end{align*}

where each of the four symbols in the new substitution represents a two-letter-word from the original (respectively, $0,1,2,3$ correspond to $[ad],[cb],[cd],[ab]$). By inspection, we see that the minimal sets of this new substitution are disjoint, and thus its coincidence partition is given by $\mathcal{P}_{\tilde{\theta}}=\{\{0,1\},\{2,3\}\}$, which means that the map $\tilde{F}\colon X_{\tilde{\theta}}\to X_{\tilde{\eta}}$, whose local rule is a code and  sends $0,1$ to $A$ and $2,3$ to $B$, is a factor map to the aperiodic, primitive, almost automorphic substitution shift given by:

\begin{align*}
              \tilde{\eta}\colon A & \mapsto B A A A A A B\\
              B & \mapsto B A A B A A B
\end{align*}

As in the construction description after Proposition
\ref{lem:almost-automorphic-height.}, 
because the original substitution $\theta$ has height $2$, to find an almost automorphic factor of $X_\theta$ we introduce a height-2 suspension $\eta$ of  $\tilde{\eta}$ by ``splitting'' each symbol into two, moving from e.g. $A\mapsto BAAAAAB$ to $Aa\mapsto BbAaAaAaAaAaBb$, which is a concatenation of two length $7$ words. The new substitution, almost automorphic by construction, is given by:

\begin{align*}
              \eta\colon A & \mapsto B b A a A a A\\
              B & \mapsto B b A a A a B\\
              a & \mapsto a A a A a B b\\
              b & \mapsto b A a A a B b.
\end{align*}

The previously defined map $\tilde{F}$ induces a factor map $F\colon X_\theta\to X_\eta$. To define it explicitly, we use the fact that each element of $\{0,1,2,3\}$ corresponds to a two-letter word in $X_\theta$ and is mapped to either $A$ or $B$, which also corresponds to the two-letter words $Aa$ or $Bb$ in $X_\eta$, so we expect $F$ to map any instance of, say, $ad$ in some $x\in X_\theta$ to the word $Aa$ in the corresponding $F(x)\in X_\eta$. We can accomplish this by giving $F$ left- and right-radius $1$; accordingly, its local rule will be:

\begin{align*}
              aba &\mapsto b &
              adc &\mapsto a \\
              bab &\mapsto B &
              bad &\mapsto A \\
              cba &\mapsto a &
              cdc &\mapsto b \\
              dcb &\mapsto A &
              dcd &\mapsto B.
\end{align*}

     \end{example}

\section{ Factoring onto a bijective substitution}\label{sec:bijective-1}

In this section, we characterise, using the semigroup $S_\theta$, when a substitution shift has a bijective substitution shift as a factor. First we restrict to the case where the factor map preserves the fixed point fibre. As some of our results require trivial height, while others don't, we continue to specify what is needed at each step.

\begin{thm} \label{thm-main-bijective-1} Let $\theta$ be a constant length
substitution  with na\"{i}ve column number $c$. 
The following are equivalent:
\begin{enumerate}
\item $S_\theta$ has a unique minimal left ideal.
\item
There is a bijective substitution $\eta$ on a $c$-letter alphabet which is an inner encoding of  $\theta$.
\end{enumerate}
\end{thm}

The above result is obvious in the case that the column number $c$ of $\theta$ equals $1$, because then $\eta$ is the unique length-$\ell$ substitution on a one-letter alphabet, and $S_\theta$  always has a unique minimal left ideal.

\begin{definition}
Let $(X,\sigma)$ and $(Y,\sigma)$ be two shifts with a common  rotation factor $(\mathcal G,R)$. Let $\pi_X: X\rightarrow \mathcal G$ and $\pi_Y: Y\rightarrow \mathcal G    $ be  respective equicontinuous factor maps.
We say that the factor map $F:X\rightarrow Y$ is {\em almost injective} for $(\mathcal G,R)$, if there is $g\in \mathcal G$ such that the restriction of $F$ to the fibre ${\pi_X}^{-1}(g)$ is injective. If we don't specify the equicontinuous factor then we take it to be the maximal common one.
\end{definition}

The following result shows that the bijective substitution $\eta$ of the last theorem has to be aperiodic and that the  factor map between $(X_\theta, \sigma)$ and $(X_\eta, \sigma)$   has to be almost injective if $\theta$ has na\"{i}ve column number $c>1$.  While we do not need any assumptions on height in Theorem  \ref{thm-main-bijective-1}, we do assume that  $\theta$ has trivial height for Theorems \ref{thm-main-bijective-2-a.e.} and  \ref{thm-bijective-3}.
 This is because in general the existence of a bijective factor for a substitution is not linked to the  existence of a bijective factor for its pure base; see Example \ref{ex:bijective-height}. The issue here is that we characterise the existence of a bijective factor in terms of the semigroup $\mathcal S_\theta$, and the relationship between this semigroup and that of the pure base of $\theta$ is not clear.

\begin{thm} \label{thm-main-bijective-2-a.e.} Let $\theta$ be a primitive, aperiodic length-$\ell$  substitution over the alphabet $\Aa$,  of trivial height. Suppose  that $\theta$ has na\"{i}ve column number $c>1$, and 
 that $(\eta,\beta)$ is  inner encoded by $\theta$, with $\eta$ bijective.
 Then $\eta$ is primitive and aperiodic, and the factor map induced by $\beta:X_\theta\to X_\eta$ is almost injective. 
 \end{thm}

\subsection{More preliminaries from semigroup theory} 
In order to prove the above theorems we need we analyse further sub-semigroups of the semigroup $\Ff(X)$ of maps from $X\to X$. Recall that the partition defined by a map 
$f:X\to Y$ is $\Pp_f =\{f^{-1}(y) | y\in Y\}$.
\begin{definition}
We say that a map $g:X\to X$ {\em preserves a partition} $\Pp\subset\Pp(X)$ if $g^{-1}(\Pp)\subset \Pp$.
\end{definition}
Stated differently, let  $\Pp = \{A_i|i\in I\}$, then $g$ preserves $\Pp$ if for all $i\in I$ there is a unique $j\in I$ such that $g^{-1}(A_i)=A_j$. Note that $g$ does not necessarily preserve $\Pp_g$.

\begin{lem} \label{lem-bij} Let $g:X\to X$ preserve a partition $\Pp$. Then
{$\left.g^{-1}\right|_{\Pp}$}  is injective and hence bijective if $\Pp$ is finite.
\end{lem}
\begin{proof} Let $A,B\in\Pp$. Suppose $g^{-1}(A)=g^{-1}(B)$ which means $A\cap \im g = B\cap \im g$. Since $A$ and $B$ are either equal or have empty intersection, $A\cap \im g = B\cap \im g$ is the case if $A= B$ or $A\cap \im g=B\cap\im g=\emptyset$. But $A\cap \im g =\emptyset$ means $g^{-1}(A)=\emptyset$, a possibility which is excluded, as a partition does not contain the empty set.
\end{proof}
\begin{lem}\label{lem-preserve}
Let $g:X\to X$ and $f:X\to Y$. If $\Pp_f=\Pp_{f\circ g}$ then $g$ preserves $\Pp_f$.
\end{lem}
\begin{proof}
By assumption
$$\{f^{-1}(y) | y\in Y\} = \{g^{-1}(f^{-1}(y)) | y\in Y\}$$
which says exactly that $g^{-1}(\Pp_f) = \Pp_f$.
\end{proof}
A map $p:X\to X$ is an idempotent if and only if it preserves its partition and maps each member $A\in \Pp_p$ to a single point in $A$. 

\begin{lem}\label{lem-preserve-0} 
Let $p:X\to X$ be an idempotent. If $p$ preserves $\Pp$ and $|\Pp|=|\Pp_p|<+\infty$ then $\Pp_p=\Pp$. 
\end{lem}
\begin{proof}
As $p$ maps each member $A\in \Pp_p$ to a single point in $A$, $\Pp_p$ is the finest partition preserved by {$p$}. Hence if $p$ preserves $\Pp$ then it preserves the partition generated by $\Pp$ and $\Pp_p$ which can't be finer than $\Pp_p$. Therefore  $|\Pp|=|\Pp_p|<+\infty$ implies $\Pp_p=\Pp$. 
\end{proof}
\begin{lem}\label{lem-preserve-1} 
Let $p,q:X\to X$ two idempotents. 
If $p$ preserves $\Pp_q$ then $qp = q$. \end{lem}
\begin{proof}
Suppose $p$ preserves $\Pp_q$. 
Then it also preserves the partition generated by $\Pp_p$ and $\Pp_q$. 
As $\Pp_p$ is the finest partition preserved by $p$, 
$\Pp_q$ must be coarser than $\Pp_p$ {or $\Pp_q=\Pp_p$  }. It follows that $\left.p^{-1}\right|_{\Pp_q}$ is the identity map from $\Pp_q$ to itself. This implies $q^{-1}=p^{-1}q^{-1}$ hence $qp=q$.
\end{proof}

The condition of the following lemma is satisfied for all compact right topological semigroups and so in particular for any finite semigroup.
\begin{lem}\label{lem-Ll}
Let $S$ be a sub-semigroup of $\Ff(X)$ which admits a kernel which contains an idempotent. The following are equivalent.
\begin{enumerate}
\item All elements of $S$ preserve the partition defined by the idempotent.
\item $S$ has a unique minimal left ideal, i.e., the kernel is left simple.
\end{enumerate}
\end{lem}
\begin{proof} $1\Rightarrow 2$. Let $q\in S$ be a minimal idempotent such that all elements of $S$ preserve $\Pp_q$. Let $p\in S$ be another minimal idempotent.
By assumption $p$ preserves $\Pp_q$. By Lemma~\ref{lem-preserve-1} we have $qp=q$. Hence $q$ lies in the minimal left ideal generated by $p$. Since all idempotents of a minimal left ideal generate the same left ideal, therefore $S$ has a unique minimal left ideal.

$2\Rightarrow 1$. Let $p$ be an idempotent in the unique minimal left ideal $L$. Let $s\in S$. Then $ps$ lies in the kernel of $S$ which coincides with $L$. By Lemma~\ref{lem-partition1}, $\Pp_{ps}=\Pp_p$. By Lemma~\ref{lem-preserve}, $s$ preserves $\Pp_p$. 
\end{proof}

\begin{proof}[Proof of Theorem~\ref{thm-main-bijective-1}] 
Suppose first that $\Ss_\theta$ contains a unique minimal left ideal. 
By Lemma~\ref{lem-Ll} all its elements preserve a partition $\Bb = \{ A_1,\cdots, A_k\}$ defined by any element of that ideal. Note that $k$ must be the na\"{i}ve column number $c$. Define the length-$\ell$ substitution $\eta$ on the alphabet $\Bb$ 
through $\eta=\eta_0|\cdots|\eta_{\ell-1}$ where
$$\eta^{-1}_m := \left.\theta^{-1}_m\right|_{\Bb} \, ,$$
in other words, $\eta_m(A_i)$ is the unique $A_j$ which contains $\theta_m(a_i)$ for some $a_i\in A_i$. This is well-defined as $\theta_m$ preserves the partition $\Bb$.
By Lemma~\ref{lem-bij} $\eta$ is bijective. 
The code $\beta$ is given by $a\mapsto A_i$ for all $a\in A_i$.

To prove the converse of the statement, 
suppose now there is a bijective substitution $\eta=\eta_0|\cdots|\eta_{\ell-1}$ on a $c$-letter alphabet $\Bb$ and a map $\beta:\Aa\to \Bb$ such that  
$$\eta_m\circ \beta = \beta\circ\theta_m$$
for all $m=0,\cdots,\ell-1$. Since $\eta$ is bijective this implies that the 
partition $\Pp_\beta$ defined by $\beta$ must be the same as that defined by each $\beta\circ \theta_m$ and then also the same as that defined by any $\beta\circ f$ for any $f\in S_\theta$. By Lemma~\ref{lem-preserve} all elements of $S_\theta$ must preserve $\Pp_\beta$. 
Let $p$ be a minimal idempotent of $S_\theta$. By definition of the na\"{i}ve column rank its rank is $c$, which is also the rank of $\beta$. Lemma~\ref{lem-preserve-0} implies therefore that $\Pp_p=\Pp_\beta$. Now, we conclude with Lemma~\ref{lem-Ll} that $\Ss_\theta$ has a unique minimal left ideal. 
\end{proof}

\begin{proof}[Proof of Theorem~\ref{thm-main-bijective-2-a.e.}]
Primitivity of $\eta$ follows directly from the primitivity of $\theta$. 

We show that $\eta$ is aperiodic. We may assume that $\eta_0=\id$ (otherwise, we take a power of $\eta$). Assume, by contradiction, that $h$ is the period of $\eta$, i.e.\ all $x\in X_\eta$ satisfy $x_{n+h}=x_n$. This is only possible if $\eta^N_{kh} = \eta_0$ for all $k\geq 0$ and $N$ large enough such that \ $kh< \ell^N$.  
Furthermore, $\frac1{h}$ must be an eigenvalue of the dynamical system defined by $\eta$ and therefore also an eigenvalue of the dynamical system defined by $\theta$.  As $\theta$ has trivial height this means that  $h$ must divide $\ell^N$ for some $N$. However, as $\eta^N_{kh} = \eta_0=\id$ we then have, for all $0\leq k<\ell$
$$\id={\eta^{N+1}}_{k\ell^N} = {\eta^N}_0\eta_k = \eta_k$$
which 
contradicts the primitivity of $\eta$, as $c>1$. 

We show that the factor map $F_\beta$ induced by $\beta$ is almost injective. The common maximal equicontinuous factor is $\Z_\ell$, as the height of $\theta$ is trivial. The regular fibres of $\theta$ have the same size as those of $\eta$. 
The factor map induced by $\eta$ maps fibres surjectively onto fibres. Also, it  must map at least one regular fibre to a regular fibre.  This is because  these fibres are a set of Haar probability measure one in $\Z_\ell$. Thus the  restriction  of $F_\beta$ to that fibre must be bijective.  
\end{proof}

\noindent
{\bf Remark.} The following shows that our characterisation of when a substitution allows for an inner encoded bijective substitution is stable under taking powers. 
  \begin{lem}\label{lem:semigroup-power}  Let $\theta$ be a constant length
    substitution.
 Then $S_{\theta^n}$ has a unique minimal left ideal if and only if $S_\theta$ has a unique minimal left ideal.
  \end{lem}
 \begin{proof}
 We have $S_{\theta^n} \subset S_\theta$, and also, since $\theta$ and $\theta^n$ have the same column number $c$, $S_{\theta^n}^{(c)} \subset S_\theta^{(c)}$.
 If $S_\theta$ has two (disjoint) minimal left ideals $L_1$ and $L_2$, then $L_1\cap S_{\theta^n}$ and $L_2\cap S_{\theta^n}$ are disjoint minimal left ideals of $S_{\theta^n}$. This shows the direction $\Rightarrow$. The other direction is actually a consequence of Theorem~\ref{thm-main-bijective-1}. If $S_\theta$ has a unique minimal left ideal then there is a bijective substitution $\eta$ and a code $\beta$ such that $\beta\circ\theta_m = \eta_m\circ\beta$. This implies $\beta\circ{\theta^n}_m = {\eta^n}_m\circ\beta$ for any $n$  and any $m \leq \ell^{n}-1$. Hence $\eta^n$ is a bijective substitution which is an inner encoding of $\theta^n$. Clearly $\eta^n$ is bijective and hence by Theorem~\ref{thm-main-bijective-1}, $S_{\theta^n}$ must have a unique minimal left ideal.
 \end{proof}

 \begin{example}
Bijective substitution shifts have a single irregular fibre modulo $\mathbb{Z}$, with respect to  their maximal equicontinuous factor.
              Thus,  it is natural to ask if a substitutive subshift with a bijective factor has only finitely many singular fibres. This is false.
                Consider
              \begin{align*}
                             \theta\colon 0 &\mapsto 021 \\
                             1 &\mapsto 130 \\
                             2 &\mapsto 201 \\
                             3 &\mapsto 310
              \end{align*}
              which is easily checked to be primitive, aperiodic and of height $1$.
              Let $\pi$ be the code
              \begin{align*}
                             \pi\colon 0,2 &\mapsto a\\
                             1,3 &\mapsto b 
                                          \end{align*}
                        It is not hard to see that    $(\eta, \pi)$ is an inner encoding of $X_\theta$ 
                                             where $\eta$ is the bijective, aperiodic substitution given by:
              \begin{align*}
                         \eta\colon a &\mapsto aab \\
                             b & \mapsto bba.
              \end{align*}
             By Theorem  \ref{thm-main-bijective-1},  $S_\theta$ has a unique minimal left ideal. Indeed, all elements of this minimal left ideal have the common partition $\{ \{0,2\}, \{1,3\}\}$. Note that any element from $\mathbb{Z}_3$ which does not contain the digit $2$  is an irregular fibre. 
           \end{example}
 
 \subsection{Factors which preserve the fixed point fibre}

 In Theorem \ref{thm-main-bijective-1}, we characterised the substitution shifts which admit a bijective inner encoding. In this section, we extend this to characterising substitution shifts which admit a bijective substitution shift as a factor via a factor map which preserves the fixed point fibre.

 \begin{lem}\label{lem:semigroup-collaring}
 Let $n=(-l,r)$. If 
 $S_{\theta^{(n)}}$ has a unique minimal left ideal then $S_\theta$ has a unique minimal left ideal.
  \end{lem}
\begin{proof}  
As  $\theta$ is an inner encoding of the collared substitution $\theta^{(n)}$,  the semigroup of the collared substitution $S_\theta$ is a homomorphic image of  $S_{\theta^{(n)}}$.         Now Lemma~\ref{lem-epi-ker} implies the result.

\end{proof}

\begin{thm}\label{thm-main-bijective-2}
Let $\theta$ be a primitive aperiodic length-$\ell$ substitution with na\"{i}ve column number $c$.
The following are equivalent.
\begin{enumerate}
\item $(X_\theta, \sigma)$ factors onto a bijective substitution shift $(X_\eta,\sigma)$ via a  factor map $F$  which preserves the fixed point fibre.
\item 
$S_\theta$ has a unique minimal left ideal.
\end{enumerate}
\end{thm}
\begin{proof}
First suppose that $F:X_\theta\rightarrow X_\eta$ is a factor map which  preserves the fixed point fibre and with $\eta$ bijective.
By Proposition \ref{prop:radius-1} there is an $n=(-l,r)$, with $0\leq l,r\leq 1$ such that $(X_\eta, \sigma)$ is a factor of $(X_{\theta^{(n)}    },\sigma)$ via  a radius zero factor map $F_\tau:    X_{\theta^{(n)}    }\rightarrow   X_\eta      $ where  $F = F_\tau\circ F_{\imath}^{-1}$. 
Now by Lemma \ref{lem:right-kappa-value}, $F_{\imath}$ preserves the fixed point fibre, and by assumption, so does $F$. Hence $F_\tau$ preserves the fixed point fibre. Finally by Lemma \ref{lem:code-factor-kappa} $\eta$ is an inner encoding of $\theta^{(n)}$.
By Theorem~\ref{thm-main-bijective-1}, $S_{\theta^{(n)}}$  has a unique minimal left ideal. 
By Lemma~\ref{lem:semigroup-collaring}, $S_\theta$ has a unique minimal left ideal. 

The converse follows directly from Theorem~\ref{thm-main-bijective-1}. 
\end{proof}
  We remark that Theorems~\ref {thm-main-bijective-1} and \ref{thm-main-bijective-2} imply that all factor maps to a bijective substitution which fix the fixed point fibre must  have radius zero.

%
%

\subsection{Bijective factors which do not preserve the fixed point fibre}

In this section we extend Theorem \ref{thm-main-bijective-2} to  substitutions which have a bijective substitution factor  via a factor map
which does not send fixed points to fixed points. 
 To do this we need to recall a little more information on the arithmetic information that factor maps encode. 
\subsection{Factors of substitution shifts and their $\kappa$-values}

We describe factor maps $F$ of substitution shifts by associating to each of them an  $\ell$-adic integer which describes how $F$ translates fibres of an equicontinuous factor. This element of $\Z_\ell$ is determined by where $F$ sends fixed points.

Let $X$, $Y$ be minimal shift spaces with a  common equicontinuous factor  $ (\mathcal G,R)$. Let $\pi_X :X\rightarrow \mathcal G$ and $  \pi_Y :Y\rightarrow \mathcal G  $ be  corresponding equicontinuous factor maps.
 Let $\Fac(X,Y)$ be the collection of factor maps from $(X,\sigma)$ to $(Y,\sigma)$, and let 
 $\Conj(X,Y)$ be the set of injective  factor maps from $(X,\sigma)$ to $(Y,\sigma)$.

 \begin{thm}\label{thm:ethan}  \cite[Theorem 3.3]{coven-quas-yassawi}.
Let  $(X,\sigma)$  and $(Y,\sigma)$ be  infinite minimal 
shifts.
Suppose that the group rotation $(\mathcal G,R)$ is the maximal 
equicontinuous factor of both $(X,\sigma)$ and $(Y,\sigma)$ and fix the factor maps
$\pi_X:X\rightarrow \mathcal G$ and $\pi_Y:Y\rightarrow \mathcal G$.
Then there is a map $\kappa:  \mbox{Fac}(X,Y) \rightarrow \mathcal G$
such that \[\pi_Y(F(x))=\kappa(F)+\pi_X(x) \] for all 
$x\in X$ and $F  \in \mbox{Fac}(X,Y)$. Also 
\begin{enumerate}
\item
 if $(Z,\sigma)$ is another 
shift which satisfies the assumptions on $(X,\sigma)$, then
$\kappa(G\circ F)=\kappa(G)+\kappa(F)$ for 
$F \in \mbox{Fac} (X,Y)$,  $G \in
\mbox{Fac}(Y,Z)$, and 
\item if $\min_{g\in \mathcal G}|\pi_X^{-1}(g)|=\min_{g\in \mathcal G}|\pi_Y^{-1}(g)|=c<\infty$,
then
 for each $F\in \mbox{Fac}(X,Y)$, we have  
\[ \{ z \in \Z_{r}: |\pi_Y^{-1}(z)|>c \}\subset
\{z \in \Z_{r}: |\pi_X^{-1}(z)|>c \} + \kappa(F),
\]
and  
$\kappa$ is at most  $c$-to-one.
\end{enumerate}
\end{thm}

 We call $\kappa(F)$ the {\em $\kappa$-value of $F$}.
 In fact $\kappa: \Fac(X,Y)\rightarrow \mathcal G$  is defined by $\kappa (F):= \pi_Y(F( x))- \pi_X(x)$, where this quantity does not depend on $x$ for minimal systems. It is important to note that the $\kappa$-value of a map depends on the choice of maximal equicontinuous factor maps.

If $F:X_\theta \rightarrow X_\eta$ is a factor map between two length-$\ell$ substitution shifts, then  $\kappa(F)=0$ if and only if $\theta$-fixed points are mapped to $\eta$-fixed points, i.e., if 
$F$
 preserves the fixed point fibre.

The following tells us   that 
$\kappa$-values of conjugacies between two length-$\ell$ shifts are constrained. {We say that $z\in \Z_\ell $ is {\em rational} if it is eventually periodic. This naming follows from the fact that if $z\in \Z_\ell$ is eventually periodic, then it is the $\ell$-adic expansion of a rational number.} 
\begin{prop}\label{prop:upper_bound_conjugacy} 
Let $\theta$  and $\theta'$ be  primitive, aperiodic length-$\ell$ substitutions, 
with $\theta$ of trivial height. If $F\in \mbox{Fac}(X_\theta,X_{\theta'})$, 
then  
$\kappa (F)$ is  rational. 
\end{prop}

\begin{proof}

In the case where the factor map is a conjugacy, the statement is \cite[Proposition 3.24]{coven-quas-yassawi}.
We will suppose that $F=F_\tau$ where $\tau:\mathcal A_\theta\rightarrow \mathcal A_{\theta'}$, as otherwise the usual combination of Proposition \ref{prop:radius-1} and Theorem \ref{th:factor-bijective} allows us to reduce to this case.

 Theorem \ref{th:factor-bijective} tells us that there is a pair $(\eta, \beta)$ which is inner encoded by $\theta$, and a code $\tau':\mathcal A_{\eta}\rightarrow \mathcal A_{\theta'} $ such that 
$F_{\tau '}:X_\eta\rightarrow X_{\theta'}$ is a conjugacy.  Again by \cite[Proposition 3.24]{coven-quas-yassawi},
$\kappa (F_{\tau'}) $ is rational. But  also, $\tau = \tau'\circ \beta$, and since $(\theta', \beta)$ is inner encoded by $\theta$, then by Lemma \ref{lem:code-factor-kappa}, $\kappa (F_\beta)=0$.  Thus $\kappa (F_{\tau})  = \kappa (F_{\tau'}) $ is also rational.

\end{proof}

If $F:X_\theta\to X_\eta$ is a factor map onto a bijective substitution shift with $\kappa(F)=m\in \Z$, then $\tilde F:=\sigma^{-m}\circ F: X_\theta\to X_\eta$ is also a factor map onto a bijective substitution and, as  $\kappa(\tilde F)=0$, we can apply Theorem~\ref{thm-main-bijective-2} to conclude that $S_\theta$ has a unique left minimal ideal. In other words, Theorem~\ref{thm-main-bijective-2} extends verbatim to the case in which the kappa-value of the factor map is an integer.

The following example shows that there exist factor maps $F$ where  $\kappa(F)\not\in \Z$.
 \begin{example}\label{ex:positive-kappa-value}
Take
 the following two  substitutions:
\begin{align*}
     \eta \colon a & \mapsto abcba, & \theta\colon 0& \mapsto 35203\\
     b & \mapsto bcacb, & 1 &\mapsto     35214\\  
     c & \mapsto cabac, & 2 &\mapsto   41520 \\   
     & & 3 &\mapsto    41534 \\ 
     & & 4 &\mapsto 02140\\ 
     & & 5 &\mapsto 02153 
\end{align*} 
The words of length two in $\mathcal L_\eta$ are
\[   \{ab,ac,ba,bc,ca,cb\},\]
and by mapping $(ab),(ac),(ba),(bc),(ca),(cb)$ to $0,1,2,3,4,5$, respectively, 
it can be verified that $\theta=\eta^{(+1)}$. Thus $X_\theta$ factors bijectively onto the bijective substitution shift $X_\eta$.
However, it can be verified that $\Ss_{\theta}$ does {\em not} have a unique minimal left ideal.
Indeed, it has two minimal idempotents, one generated by the partition $\{ \{ 0,1\}, \{2,3 \},\{4,5 \}\}$, and the other generated by the partition  $\{ \{ 0,5\}, \{1,3 \},\{2,4 \}\}$; see Lemma \ref{lem-partition1}.
This does not contradict Theorem \ref{thm-main-bijective-2} because we will see that  the radius zero conjugacy $F_\tau:X_{\theta}\rightarrow X_\eta$, which is given by the  local rule $\tau$:
\begin{align*}
	\tau:0,1&\mapsto a, \\ 2,3&\mapsto b, \\ 4,5&\mapsto c
\end{align*}
is a factor map with nonzero $\kappa$-value. In particular,  Lemma \ref{lem:kappa values} will tell us that    $\kappa( F_\tau)= -1/4= -1/(\ell -1)$.
\end{example}

We now show how to  ``correct" such factor maps  to obtain factor maps with $\kappa$-value $0$.

Recall the $k$-shifted extension of $\theta$ from Definition \ref{def:shifting}.

\begin{lem}\label{lem:kappa values}
  Let $\theta : \mathcal{A} \rightarrow \mathcal{A}^{\ell}$ be an aperiodic primitive  length-$\ell$ substitution of trivial height, and let $\zeta:=\theta^{(+ k)}$ be its $k$-shifted extension,
   with $0 \leqslant k < \ell$.
  Then there is a conjugacy given by a code $F :
  X_{\zeta}  \rightarrow X_{\theta} $ and  $\kappa (F) = \frac{k}{
  1-\ell}$.
  \end{lem}

\begin{proof}
Let $F^{-1}:X_\theta \rightarrow X_\zeta$ be the natural conjugacy with right radius $1$ and left radius $0$, where the valid word $ab$ is mapped to the letter  $(a,b)$. Then iteration of 
\[ \zeta\circ F^{-1} = F^{-1} \circ \sigma^k\circ \theta \]
gives
\[ \zeta^n \circ F^{-1} = F^{-1} \circ (\sigma^k\circ \theta)^n= F^{-1}\circ \sigma^{k\frac{1-\ell^n}{1-\ell}} \circ \theta^n. \]
Applying this identity to a word of length $2$, and letting $n\rightarrow \infty$, we obtain, for some fixed points $u$ and $v$ of $\theta$ and $\zeta$ respectively, 
\[ v = F^{-1} \circ  \lim_{n\rightarrow \infty}  \sigma^{k\frac{1-\ell^n}{1-\ell}} (u), \]
so that $\kappa( F^{-1} ) = \frac{-k}{1-\ell}$. The result follows.
\end{proof}

Remember   that $F^{-1}$ is not necessarily a radius zero factor map.

\begin{cor}\label{cor:all kappa values}
Let $\theta : \mathcal{A} \rightarrow \mathcal{A}^{\ell}$ be an aperiodic primitive  length-$\ell$ substitution of trivial height,
   and let 
  $\frac{p}{q}\in \Z_\ell$ be irreducible. Then there exists
  some substitution $\zeta$ and a conjugacy
   $F : X_{\zeta} \rightarrow X_{\theta}$ that 
   satisfies $\kappa (F)= \frac{p}{q}$. 
   \end{cor}

\begin{proof}
  As shown in Lemma \ref{lem:kappa values}, if $0\leq k<\ell$ and  $\zeta = \theta^{(+k)}$, then $\kappa (F) = \frac{k}{1 -
  \ell}$. Similarly, if we replace $\theta$ by $\theta^m$ for some value of
  $m$ and take $\zeta = (\theta^m)^{(+k)}$, the corresponding factor map has
  $\kappa$-value $\frac{k}{1 - \ell^m}$, where we can take $k$ to be any value
  between $0$ and $\ell^m - 1$.
  
  To prove the general statement, suppose first that $- q < p \leqslant 0$, and that $p$ and $q$ have no common divisors. As we assume that $p/q\in \Z_\ell$, then $q$ and $\ell$ are coprime, so that
    $\ell^{\varphi (q)} \equiv 1 \pmod{q}$, where
  $\varphi$ is Euler's totient function. Thus
  $\ell^{\varphi (q)} - 1 = hq$ for some $h > 0$. Since
  $- q < p \leqslant 0$, we have $0 \leqslant - hp < hq < \ell^{\varphi (q)}$, so we can take $\zeta = (\theta^{\varphi (q)})^{( + (-hp))}$,  and the natural conjugacy $F : X_{\zeta} \rightarrow
  X_{\theta}$ will satisfy:
  \[ \kappa (F) = \frac{-hp}{1 - \ell^{\varphi (q)}} =
     \frac{-hp}{-hq} =
     \frac{p}{q}, \]
  as desired. 
  
  For the general case, it is enough to note that $\frac{p}{q} = M
  + \frac{p_0}{q}$ for some integer $M $ and $- q < p_0
  \leqslant 0$, where $p_0$ will also be coprime to $q$, so it suffices to
  take $\zeta = (\theta^{\varphi (q)})^{(+(- hp_0))}$,
  as above, and $\sigma^M \circ F$ as the desired conjugacy.
\end{proof}

Define, for a substitution of trivial height with column number $c$,
 \[ 
\mathcal F_\theta:= \{  w=w_k \ldots w_1 \in (\Z / \ell\Z)^+: 
|\theta_{w_1}\circ \ldots \circ\theta_{w_k}(\mathcal A)|=c   \}.
\]

Note that if $ \mathcal F_\theta$ contains a length $k$ word, then $\theta^k$ has a column with $c$ elements.
\begin{lem}\label{lem:minimal}
Let $\theta$ be a constant length substitution.
Let $(\eta,\beta)$ be an inner encoding of $\theta$. The length $j(\eta)$ of the smallest word in $\Ff_\eta$ is bounded by the length $j(\theta)$ of the smallest word in $\Ff_\theta$.
\end{lem}
\begin{proof} Let $U_\theta = \{\im f:f\in \ker S_\theta\}$. We first show that $U_\eta = \beta(U_\theta)$. Indeed, by Lemma~\ref{lem-epi-ker} we have $\ker S_\eta = \beta(\ker S_\theta)\beta^{-1}$ and hence
$$U_\eta = \{\im f:f\in \ker S_\eta\} = \{\im \beta\circ g : g \in \ker S_\theta\}.$$
By definition, $w_k \ldots w_1$ is a word in $\Ff_\theta$ if and only if 
$\theta_{w_1}\circ \ldots \circ\theta_{w_k}(\mathcal A)\in U_\theta$. Now, by the above, 
$$\eta_{w_1}\circ \ldots \circ\eta_{w_k}(\mathcal B)= \beta \circ\theta_{w_1}\circ \ldots  \circ\theta_{w_k}(\mathcal A)\in \beta(U_\theta) = U_\eta.$$  
This implies $\Ff_\theta\subset\Ff_\eta$ and hence $j(\eta)\leq j(\theta)$.
\end{proof}

\begin{lem}\label{lem:kappa-rational}
Let $\theta : \mathcal{A} \rightarrow \mathcal{A}^{\ell}$ be an aperiodic primitive  length-$\ell$ substitution of trivial height,
and let 
 $F:X_\theta\rightarrow X_\zeta$ be a factor map.
If   $\mathcal F_\theta$ contains a  word of length $j$, then $n\kappa (F) \in \Z$ for some $0\leq n\leq (\ell -1) ( \ell^{j}-1)$.
\end{lem}

\begin{proof}

In the case where the factor map is a conjugacy, the statement is \cite[Proposition 3.24]{coven-quas-yassawi}.
Otherwise, we can suppose that $F=F_\tau$ where $\tau:\mathcal A_\theta\rightarrow \mathcal A_{\zeta}$ (invoking the usual combination of Proposition \ref{prop:radius-1} and Theorem \ref{th:factor-bijective}  to reduce to this case).

 Theorem \ref{th:factor-bijective} tells us that there is a pair $(\eta, \beta)$ which is inner encoded by $\theta$, and a code $\tau':\mathcal A_{\eta}\rightarrow \mathcal A_{\zeta} $ such that 
$F_{\tau '}:X_\eta\rightarrow X_{\zeta}$ is a conjugacy,  and $\kappa(F_\tau)= \kappa(F_{\tau'})$. Then again by \cite[Proposition 3.24]{coven-quas-yassawi}, the denominator of the reduced form of $\kappa(F_{\tau'})$ is at most $(\ell -1) ( \ell^{j}-1)$, where $j=j(\eta)$ is the length of the smallest word in $\mathcal F_\eta$.  Now by Lemma \ref{lem:minimal}, we have that $j(\eta)\leq j(\theta)$ . Thus
the denominator of $\kappa(F_{\tau'})$ is at most $(\ell -1) ( \ell^{j(\theta)}-1)$. Since $\kappa(F_{\tau'})= \kappa(F_\tau)$, the result follows.
\end{proof}

\begin{thm}\label{thm-bijective-3}
Let $\theta$ be an aperiodic primitive constant length-$\ell$ substitution with column number $c > 1$ and trivial height. The following are equivalent:
\begin{enumerate} 
\item $X_\theta$ factors onto an aperiodic bijective length-$\ell$ substitution shift whose alphabet has c letters and such that the factor map is almost everywhere one-to-one.
\item There exists $0\leq n,k$ such that the semigroup $S_{(\theta^n)^{(+k)}}$  contains a unique minimal left ideal. \\ Moreover, if $j$ is the length of the smallest word in $\mathcal F_\theta$, then $n\leq (\ell -1)(\ell^j -1)-1$ and $0\leq k \leq \ell^n$.
\end{enumerate}

\end{thm}
\begin{proof} Let $F:(X_\theta,\sigma)\to (X_\eta,\sigma)$ where $\eta$ is bijective, defined on a $c$-letter alphabet. By 
Proposition \ref{prop:upper_bound_conjugacy},
$\kappa(F)$ is  rational, and by Lemma \ref{lem:kappa-rational}, $\kappa(F)=p/q$ where $1\leq q\leq (\ell -1)(\ell^j -1)$, where $j$ is the length of the smallest word in $\mathcal F_\eta$. Using $\varphi(q) \leq  (\ell -1)(\ell^j -1)-1$, and taking a shift if necessary,
 Corollary \ref{cor:all kappa values} tells us that there is $n\leq (\ell -1)(\ell^j -1)-1$ and $0\leq k \leq \ell^{n}$, and a conjugacy $\tilde F :(X_{(\theta^n)^{(+k)}},\sigma)\to (X_\theta,\sigma)$   such that      
$\kappa(F)+\kappa(\tilde F)\in \Z$. As $\kappa(         F\circ \tilde F)\in \Z$,  then $S_{(\theta^n)^{(+k)}}$ must have a unique minimal left ideal by Theorem~\ref{thm-main-bijective-2}.

For the converse, if $S_{(\theta^n)^{(+k)}}$ has a unique minimal left ideal, then by Theorem \ref{thm-main-bijective-2}, $(X_{(\theta^n)^{(+k)}},\sigma)$
factors dynamically onto a bijective substitution shift. As $(X_{(\theta^n)^{(+k)}},\sigma)$ is conjugate to $(X_\theta, \sigma)$, the result follows.
\end{proof}

\begin{example}
We return to Example \ref{ex:positive-kappa-value}, with $\theta$ and $\eta$ defined there.
We take $\zeta:= \theta^{(+3)}$, and we claim that $\zeta$ has a unique minimal ideal. The reason why we make this choice $\theta^{(+3)}$ is that, by Lemma  \ref{lem:kappa values},  the natural conjugacy $G: X_\zeta \rightarrow X_\theta$ satisfies $\kappa(G) = -3/4$.  
Then  $\kappa(\sigma \circ F)= 3/4$,
and $F\circ \sigma\circ G: X_\zeta\rightarrow X_\eta$ has $\kappa$-value zero. 
Note though that $F \circ \sigma \circ G $ has right  radius one, so we will need to work with $X_{\zeta^{(2)}}$. Theorem \ref{thm-main-bijective-2}
 guarantees that $\mathcal S_{\zeta^{(2)}}$ has a unique minimal left ideal, and Lemma \ref{lem:semigroup-collaring}
  guarantees that so also does $\mathcal S_{\zeta}$.

To define $\zeta$, we  simultaneously list and code the  words of length two in $\mathcal L_\theta$:
\begin{align*}
A&=02, \hspace{5em}
B=03, \\
C&=14,
\hspace{5em}D=15,\\
E&=20,
\hspace{5em}F=21,\\
G&=34,
\hspace{5em}H=35,\\
I&=40,
\hspace{5em}J=41,\\
K&=52,
\hspace{5em}L=53;
\end{align*}
it then can be verified that the 3-shifted extension  $\zeta:=\theta^{(+3)}$ of $\theta$ is given by
\begin{align*}
\zeta: A &\mapsto BGJDK \\
B &\mapsto BGJDL \\
C &\mapsto CIAFC \\
D &\mapsto CIAFD \\
E &\mapsto EBHKE \\
F &\mapsto EBHKF\\
G &\mapsto GIAFC \\
H &\mapsto GIAFD \\
I &\mapsto IBHKE\\
J &\mapsto IBHKF \\
K &\mapsto LGJDK \\
L &\mapsto LGJDL
\end{align*}
\end{example}

\begin{example}\label{ex:exhaustive}
              Consider the following length $3$ substitution:
              \begin{align*}
                             \theta\colon a & \mapsto a b f\\
                             b & \mapsto a e f\\
                             c & \mapsto a b f\\
                             d & \mapsto d e c\\
                             e & \mapsto d b c\\
                             f & \mapsto d e c.
              \end{align*}
        This substitution has column number $2$, trivial height, is primitive, and it has two left minimal ideals, associated to the partitions $\mathcal{P}_{\theta_0}=\mathcal{P}_{\theta_2}=\{\{a,b,c\},\{d,e,f\}\}$ and $\mathcal{P}_{\theta_1} = \{\{a,c,e\},\{b,d,f\}\}$. In fact, this substitution is \emph{quasi-bijective} (see \cite{Lemanczyk-Muellner}) in the sense that every element in $S_\theta$ (bar the identity map) has rank $2$, and thus the only irregular fibre of this substitution corresponds to the set of fixed points. Thus there  are forbidden words of length one, and $j=1$.

              Theorem~\ref{thm-bijective-3} tells us that, if for every $0\le n\le (\ell-1)(\ell^j-1)-1$ and $0\le k <\ell^n$ the semigroup $S_{(\theta^n)^{(+k)}}$ has more than one left minimal ideal, then the shift $X_\theta$ has no non-trivial bijective factors; here, $j = 1$, so the bounds for $n$ and $k$ are $3$ and $27$, respectively.

              Via automated computation, we can determine each of the finitely many substitutions $(\theta^n)^{(+k)}$, e.g. for $n=3$ and $k=11$ we get the following:
 
              \begin{align*}
                a_{b} & \mapsto f_{d} d_{b} b_{c} c_{d} d_{e} e_{c} c_{d} d_{e} e_{c} c_{d} d_{b} b_{c} c_{a} a_{b} b_{f} f_{a} a_{b} b_{f} f_{a} a_{e} e_{f} f_{d} d_{e} e_{c} c_{d} d_{e} e_{c}\\
                a_{e} & \mapsto f_{d} d_{b} b_{c} c_{d} d_{e} e_{c} c_{d} d_{e} e_{c} c_{d} d_{b} b_{c} c_{a} a_{b} b_{f} f_{d} d_{e} e_{c} c_{d} d_{b} b_{c} c_{a} a_{b} b_{f} f_{a} a_{b} b_{f}\\
                 b_{c} & \mapsto c_{a} a_{e} e_{f} f_{a} a_{b} b_{f} f_{d} d_{e} e_{c} c_{d} d_{b} b_{c} c_{a} a_{b} b_{f} f_{a} a_{b} b_{f} f_{a} a_{e} e_{f} f_{d} d_{e} e_{c} c_{a} a_{b} b_{f}\\
                 b_{f} & \mapsto c_{a} a_{e} e_{f} f_{a} a_{b} b_{f} f_{d} d_{e} e_{c} c_{d} d_{b} b_{c} c_{a} a_{b} b_{f} f_{d} d_{e} e_{c} c_{d} d_{b} b_{c} c_{a} a_{b} b_{f} f_{d} d_{e} e_{c}\\
                  c_{a} & \mapsto f_{d} d_{b} b_{c} c_{d} d_{e} e_{c} c_{d} d_{e} e_{c} c_{d} d_{b} b_{c} c_{a} a_{b} b_{f} f_{a} a_{b} b_{f} f_{a} a_{e} e_{f} f_{d} d_{e} e_{c} c_{a} a_{b} b_{f}\\
                  c_{d} & \mapsto f_{d} d_{b} b_{c} c_{d} d_{e} e_{c} c_{d} d_{e} e_{c} c_{d} d_{b} b_{c} c_{a} a_{b} b_{f} f_{d} d_{e} e_{c} c_{d} d_{b} b_{c} c_{a} a_{b} b_{f} f_{d} d_{e} e_{c}\\
                  d_{b} & \mapsto c_{a} a_{e} e_{f} f_{a} a_{b} b_{f} f_{a} a_{b} b_{f} f_{a} a_{e} e_{f} f_{d} d_{e} e_{c} c_{a} a_{b} b_{f} f_{a} a_{e} e_{f} f_{d} d_{e} e_{c} c_{d} d_{e} e_{c}\\
                  d_{e} & \mapsto c_{a} a_{e} e_{f} f_{a} a_{b} b_{f} f_{a} a_{b} b_{f} f_{a} a_{e} e_{f} f_{d} d_{e} e_{c} c_{d} d_{e} e_{c} c_{d} d_{b} b_{c} c_{a} a_{b} b_{f} f_{a} a_{b} b_{f}\\
                  e_{c} & \mapsto f_{d} d_{b} b_{c} c_{d} d_{e} e_{c} c_{a} a_{b} b_{f} f_{a} a_{e} e_{f} f_{d} d_{e} e_{c} c_{a} a_{b} b_{f} f_{a} a_{e} e_{f} f_{d} d_{e} e_{c} c_{a} a_{b} b_{f}\\
                             e_{f} & \mapsto f_{d} d_{b} b_{c} c_{d} d_{e} e_{c} c_{a} a_{b} b_{f} f_{a} a_{e} e_{f} f_{d} d_{e} e_{c} c_{d} d_{e} e_{c} c_{d} d_{b} b_{c} c_{a} a_{b} b_{f} f_{d} d_{e} e_{c}\\
                               f_{a} & \mapsto c_{a} a_{e} e_{f} f_{a} a_{b} b_{f} f_{a} a_{b} b_{f} f_{a} a_{e} e_{f} f_{d} d_{e} e_{c} c_{a} a_{b} b_{f} f_{a} a_{e} e_{f} f_{d} d_{e} e_{c} c_{a} a_{b} b_{f}\\
                                    f_{d} & \mapsto c_{a} a_{e} e_{f} f_{a} a_{b} b_{f} f_{a} a_{b} b_{f} f_{a} a_{e} e_{f} f_{d} d_{e} e_{c} c_{d} d_{e} e_{c} c_{d} d_{b} b_{c} c_{a} a_{b} b_{f} f_{d} d_{e} e_{c}\\
              \end{align*}
 where each symbol $u_v$ corresponds to a $2$-letter word in $\theta$. Computing the corresponding semigroup for each, we observe that in every case it has between $2$ and $4$ minimal left ideals; for example, in the above case there are four minimal left ideals given by the following partitions:

              \begin{align*}
                             \mathcal{P}_1 &= \{\{b_f,e_f, e_c, f_a, c_d,c_a \},\{f_d,d_b,d_e,b_c,a_b,a_e\}\}\\
                             \mathcal{P}_2 &= \{\{b_f,f_d, d_e, e_c, f_a,a_e \},\{d_b,e_f,b_c,a_b,c_d,c_a\}\}\\
                             \mathcal{P}_3 &= \{\{b_f,e_f, e_c, a_b, a_e,c_a \},\{f_d,d_b,d_e,b_c,f_a,c_d\}\}\\
                             \mathcal{P}_4 &= \{\{b_f,e_f, d_e, b_c, a_b,c_d \},\{f_d,d_b,e_c,f_a,a_e,c_a\}\}.
              \end{align*}

              Hence, $X_\theta$ cannot factor onto any bijective substitution. \end{example}

 We end with an example that shows that  if a substitution shift has a bijective shift factor, this does not imply that its pure base has a bijective shift factor.
   \begin{example}\label{ex:bijective-height}

                             Consider the following substitution on $\{a,\bar{a},b,c,d,e,f\}$:

                             \begin{align*}\theta\colon a & \mapsto a d c\\
                                           \bar{a} & \mapsto \bar{a} d c\\
                                           b & \mapsto b e a\\
                                           c & \mapsto c f b\\
                                           d & \mapsto d \bar{a} e\\
                                           e & \mapsto e b f\\
                                           f & \mapsto f c d.                                        
                             \end{align*}

                             It is routine to check that this substitution is primitive and has height $2$.
                                                        We can define a  factor map $F_\tau:X_\theta \rightarrow X_\eta$, 
                             which has local rule $\tau: \{a,\bar{a},b,c,d,e,f\}\rightarrow \{a,b,c,d,e,f\}$, with $\tau(a)= \tau(\bar{a})$, and $\tau$ is the identity otherwise; $\eta$ is the resulting inner encoding.
                              It can be verified that $\eta$ is bijective.  As such, we can check that $S_\theta$ has indeed a unique minimal left ideal, satisfying the conditions of Theorem~\ref{thm-main-bijective-1}.

                             Given that both $\theta$ and $\eta$ have height $2$, one may proceed as in the proof of Proposition~\ref{lem:almost-automorphic-height.} and define a factor map between the pure bases $\tilde{\theta}$ and $\tilde{\eta}$ of both substitutions, but, in general, the property of being bijective is {\em not} a conjugacy invariant. In fact, $\tilde{\eta}$ is not a bijective substitution.  Furthermore, it cannot be taken to be conjugate to one: it can be verified that  the maximal equicontinuous factor map of $X_{\tilde{\eta}}$ has two irregular fibres modulo $\mathbb{Z}$, namely  $0$ and $\cdots 1,1,1 = -1/2$, but any subshift that is conjugate to one arising from a bijective substitution can only have one irregular fibre modulo $\mathbb{Z}$.

                             Theorem \ref{thm-bijective-3} tells us that to decide whether $\tilde{\theta}$ has a bijective factor, we need to compute the minimal left  ideals of $(\tilde \theta^n)^{(+k)}$ with $1\leq n \leq 3$ and $0\leq k \leq 27$. Via automated computation, we have verified that $\tilde{\theta}$ has no bijective substitution shift factors. (See Example \ref{ex:exhaustive} for one example of such a computation.)
                                        \end{example}

{\footnotesize
\bibliographystyle{abbrv}

\def\ocirc#1{\ifmmode\setbox0=\hbox{$#1$}\dimen0=\ht0 \advance\dimen0
  by1pt\rlap{\hbox to\wd0{\hss\raise\dimen0
  \hbox{\hskip.2em$\scriptscriptstyle\circ$}\hss}}#1\else {\accent"17 #1}\fi}

}

\end{document}